\documentclass[12pt]{amsart}
\usepackage{amsmath,amssymb,amscd}
\usepackage[usenames,dvipsnames]{color}
\usepackage{graphicx}
\usepackage{longtable}

\newtheorem{propo}{Proposition}[section]
\newtheorem{corol}[propo]{Corollary}
\newtheorem{theor}[propo]{Theorem}
\newtheorem{lemma}[propo]{Lemma}

\theoremstyle{definition}
\newtheorem{defin}[propo]{Definition}

\theoremstyle{remark}
\newtheorem{remar}[propo]{Remark}

\newcommand{\algo}[6]
{\vspace{11pt}\addtocounter{propo}{1}
\noindent{\bf Algorithm \arabic{section}.\arabic{propo}.}
{\bf #1}(#2)\\ {\it #3}.\\
{\bf Input:} #4\\
{\bf Output:} #5\\
\newcounter{#1}
\begin{list}{\textbf{\arabic{#1}.}}{\usecounter{#1}}
#6\end{list}\vspace{3pt}}

\numberwithin{equation}{section}

\newcommand{\al }{\alpha }

\newcommand{\cC }{\mathcal{C}}
\newcommand{\cF }{\mathcal{F}}
\newcommand{\Cm }{C}
\newcommand{\cm }{c}

\newcommand{\GL }{\mathrm{GL}}

\newcommand{\Homsfrom }[2][\Wg (\cC )]{\mathrm{Hom}(#2,#1)}

\newcommand{\id }{\mathrm{id}}
\newcommand{\lspan }{\mathrm{span}}
\newcommand{\ndN }{\mathbb{N}}
\newcommand{\ndQ }{\mathbb{Q}}

\newcommand{\ndZ }{\mathbb{Z}}
\newcommand{\NN }{\ndN }
\newcommand{\QQ }{\ndQ }
\newcommand{\ZZ }{\ndZ }
\newcommand{\rer }[1]{(R\re)^{#1}}
\newcommand{\rfl }{\rho }
\newcommand{\rk }{\mathrm{rk}}
\newcommand{\Ob }{\mathrm{Ob}}

\newcommand{\re }{^\mathrm{re}}
\newcommand{\rsC }{\mathcal{R}}
\newcommand{\s }{\sigma }

\newcommand{\Vol }{\mathrm{Vol}}
\newcommand{\Wg }{\mathcal{W}}
\DeclareMathOperator{\Aut}{Aut}

\DeclareMathOperator{\Hom}{Hom}

\title{Finite Weyl groupoids of rank three}

\author{M.~Cuntz}
\address{Michael Cuntz,
Fachbereich Mathematik,
Universit\"at Kaiserslau\-tern,
Postfach 3049,
D-67653 Kaiserslautern, Germany}
\email{cuntz@mathematik.uni-kl.de}

\author{I.~Heckenberger}
\thanks{I.H. is
supported by the German Research Foundation (DFG) via a Heisenberg
fellowship}
\address{Istv\'an Heckenberger, Philipps-Universit\"at Marburg,
Fachbereich Mathematik und Informatik,
Hans-Meerwein-Stra\ss e,
D-35032 Marburg, Germany}
\email{heckenberger@mathematik.uni-marburg.de}

\begin{document}

\begin{abstract}
We continue our study of Cartan schemes and their Weyl groupoids.
The results in this paper provide an algorithm to
determine connected simply connected Cartan schemes of rank three,
where the real roots form a finite irreducible root system.
The algorithm terminates: Up to equivalence there are
exactly 55 such Cartan schemes, and the number of corresponding real roots
varies between $6$ and $37$. We identify those Weyl groupoids which appear
in the classification of Nichols algebras of diagonal type.
\end{abstract}

\maketitle

\section{Introduction}

Root systems associated with Cartan matrices are widely studied structures in
many areas of mathematics, see \cite{b-BourLie4-6} for the fundaments.
The origins of the theory of root systems go back at least to the study
of Lie groups by Lie, Killing and Cartan. The symmetry of the root system is
commonly known as its Weyl group. Root systems associated with a
family of Cartan matrices appeared first in connection with Lie superalgebras
\cite[Prop.\,2.5.6]{a-Kac77}
and with Nichols algebras \cite{a-Heck06a}, \cite{a-Heck08a}.
The corresponding symmetry is not a group but a
groupoid, and is called the Weyl groupoid of the root system.

Weyl groupoids of root systems properly generalize Weyl groups. The
nice properties of this more general structure
have been the main motivation to develop an axiomatic
approach to the theory, see \cite{a-HeckYam08}, \cite{a-CH09a}.
In particular, Weyl groupoids are generated by reflections and Coxeter
relations, and they satisfy a Matsumoto type theorem \cite{a-HeckYam08}.
To see more clearly the extent of generality it would be desirable to have a
classification of finite Weyl
groupoids.\footnote{In this introduction by a Weyl groupoid
we will mean the Weyl groupoid of a connected Cartan scheme, and
we assume that the real roots associated to the Weyl groupoid form an
irreducible root system in the sense of \cite{a-CH09a}.}
However, already the appearance of a large family of examples of Lie
superalgebras and
Nichols algebras of diagonal type indicated that a classification of
finite Weyl groupoids is probably much more complicated than the
classification of finite Weyl groups. Additionally,
many of the usual classification tools
are not available in this context because of the lack of the adjoint action
and a positive definite bilinear form.

In previous work, see \cite{a-CH09b} and \cite{p-CH09a},
we have been able to determine all finite
Weyl groupoids of rank two.
The result of this classification
is surprisingly nice: We found a close relationship to the
theory of continued fractions and to cluster algebras of type $A$.
The structure of finite rank two Weyl groupoids and the associated root
systems has a natural characterization in terms of triangulations
of convex polygons by non-intersecting diagonals. In particular, there are
infinitely many such groupoids.

At first view there is no reason to assume
that the situation for finite Weyl groupoids of rank three
would be much different
from the rank two case.
In this paper we give some theoretical indications which strengthen the
opposite point of view. For example in Theorem~\ref{cartan_6} we show that the
entries of the Cartan matrices in a finite Weyl groupoid cannot be smaller
than $-7$. Recall that for Weyl groupoids there is no lower bound for the
possible entries of generalized Cartan matrices.
Our main achievement in this paper is to provide an algorithm
to classify finite Weyl groupoids of rank three. Our algorithm terminates
within a short time, and produces a finite list of examples.
In the appendix we list the root systems characterizing
the Weyl groupoids of the classification: There are
$55$ of them which correspond to pairwise
non-isomorphic Weyl groupoids. The number of positive roots in these root
systems varies between $6$ and $37$. Among our root systems are the usual root
systems of type $A_3$, $B_3$, and $C_3$, but for most of the other
examples we don't have yet an explanation.

It is remarkable that the number $37$ has a particular
meaning for simplicial arrangements in the real projective plane.
An arrangement is the complex generated by a family of straight lines not
forming a pencil. The vertices of the complex are the intersection points
of the lines, the edges are the segments of the lines between two vertices,
and the faces are the connected components of the
complement of the set of lines generating the arrangement. An arrangement is
called simplicial, if all faces are triangles. Simplicial arrangements have
been introduced in \cite{a-Melchi41}. The classification of simplicial
arrangements in the real projective plane is an open problem.
The largest known exceptional example is generated by $37$ lines. Gr\"unbaum
conjectures that the list given in \cite{a-Gruenb09} is complete. In our
appendix we provide some data of our root systems which can be used to compare
Gr\"unbaum's list with Weyl groupoids. There is an astonishing analogy between
the two lists, but more work has to be done to be able to explain the precise
relationship. This would be desirable in particular since
our classification of finite Weyl groupoids of rank three does not give any
evidence for the range of solutions besides the explicit computer calculation.

In order to ensure the termination of our algorithm,
besides Theorem \ref{cartan_6} we use a weak convexity property of certain
affine hyperplanes, see Theorem \ref{convex_diff2}: We
can show that any positive root in an affine hyperplane ``next to the origin''
is either simple or can can be written as the sum of a simple root and
another positive root.
Our algorithm finally becomes practicable by the use of
Proposition~\ref{pr:suminR}, which can be interpreted as another
weak convexity property
for affine hyperplanes. It is hard to say which of these theorems
are the most valuable because avoiding any of them makes the
algorithm impracticable (unless one has some replacement).

The paper is organized as follows. We start with two sections
proving the necessary theorems to formulate the algorithm:
The results which do not require that the rank is three are
in Section \ref{gen_res}, the obstructions for rank three in
Section \ref{rk3_obst}.
We then describe the algorithm in the next section. Finally
we summarize the resulting data and make some observations in
the last section.

\vspace{\baselineskip}

\textbf{Acknowledgement.} We would like to thank B. M{\"u}hlherr
for pointing out
to us the importance of the number $37$ for simplicial arrangements in the
real projective plane.

\section{Cartan schemes and Weyl groupoids}\label{gen_res}

We mainly follow the notation in \cite{a-CH09a,a-CH09b}.
The fundaments of the general theory have been developed in
\cite{a-HeckYam08} using a somewhat different terminology.
Let us start by recalling the main definitions.

Let $I$ be a non-empty finite set and
$\{\al _i\,|\,i\in I\}$ the standard basis of $\ndZ ^I$.
By \cite[\S 1.1]{b-Kac90} a generalized Cartan matrix
$\Cm =(\cm _{ij})_{i,j\in I}$
is a matrix in $\ndZ ^{I\times I}$ such that
\begin{enumerate}
  \item[(M1)] $\cm _{ii}=2$ and $\cm _{jk}\le 0$ for all $i,j,k\in I$ with
    $j\not=k$,
  \item[(M2)] if $i,j\in I$ and $\cm _{ij}=0$, then $\cm _{ji}=0$.
\end{enumerate}

Let $A$ be a non-empty set, $\rfl _i : A \to A$ a map for all $i\in I$,
and $\Cm ^a=(\cm ^a_{jk})_{j,k \in I}$ a generalized Cartan matrix
in $\ndZ ^{I \times I}$ for all $a\in A$. The quadruple
\[ \cC = \cC (I,A,(\rfl _i)_{i \in I}, (\Cm ^a)_{a \in A})\]
is called a \textit{Cartan scheme} if
\begin{enumerate}
\item[(C1)] $\rfl _i^2 = \id$ for all $i \in I$,
\item[(C2)] $\cm ^a_{ij} = \cm ^{\rfl _i(a)}_{ij}$ for all $a\in A$ and
  $i,j\in I$.
\end{enumerate}

  Let $\cC = \cC (I,A,(\rfl _i)_{i \in I}, (\Cm ^a)_{a \in A})$ be a
  Cartan scheme. For all $i \in I$ and $a \in A$ define $\s _i^a \in
  \Aut(\ndZ ^I)$ by
  \begin{align}
    \s _i^a (\al _j) = \al _j - \cm _{ij}^a \al _i \qquad
    \text{for all $j \in I$.}
    \label{eq:sia}
  \end{align}
  The \textit{Weyl groupoid of} $\cC $
  is the category $\Wg (\cC )$ such that $\Ob (\Wg (\cC ))=A$ and
  the morphisms are compositions of maps
  $\s _i^a$ with $i\in I$ and $a\in A$,
  where $\s _i^a$ is considered as an element in $\Hom (a,\rfl _i(a))$.
  The category $\Wg (\cC )$ is a groupoid in the sense that all morphisms are
  isomorphisms.
  The set of morphisms of $\Wg (\cC )$ is denoted by $\Hom(\Wg (\cC ))$,
  and we use the notation
  \[ \Homsfrom{a}=\mathop{\cup }_{b\in A}\Hom (a,b) \quad
  \text{(disjoint union)}. \]
  For notational convenience we will often neglect upper indices referring to
  elements of $A$ if they are uniquely determined by the context. For example,
  the morphism $\s _{i_1}^{\rfl _{i_2}\cdots \rfl _{i_k}(a)}
  \cdots \s_{i_{k-1}}^{\rfl _{i_k(a)}}\s _{i_k}^a\in \Hom (a,b)$,
  where $k\in \ndN $, $i_1,\dots,i_k\in I$, and
  $b=\rfl _{i_1}\cdots \rfl _{i_k}(a)$,
  will be denoted by $\s _{i_1}\cdots \s _{i_k}^a$ or by
  $\id _b\s _{i_1}\cdots \s_{i_k}$.
  The cardinality of $I$ is termed the \textit{rank of} $\Wg (\cC )$.
  A Cartan scheme is called \textit{connected} if its Weyl groupoid
  is connected, that is, if for all $a,b\in A$ there exists $w\in \Hom (a,b)$.
  The Cartan scheme is called \textit{simply connected},
  if $\Hom (a,a)=\{\id _a\}$ for all $a\in A$.

  Let $\cC $ be a Cartan scheme. For all $a\in A$ let
  \[ \rer a=\{ \id _a \s _{i_1}\cdots \s_{i_k}(\al _j)\,|\,
  k\in \ndN _0,\,i_1,\dots,i_k,j\in I\}\subset \ndZ ^I.\]
  The elements of the set $\rer a$ are called \textit{real roots} (at $a$).
  The pair $(\cC ,(\rer a)_{a\in A})$ is denoted by $\rsC \re (\cC )$.
  A real root $\al \in \rer a$, where $a\in A$, is called positive
  (resp.\ negative) if $\al \in \ndN _0^I$ (resp.\ $\al \in -\ndN _0^I$).
  In contrast to real roots associated to a single generalized Cartan matrix,
  $\rer a$ may contain elements which are neither positive nor negative. A good
  general theory, which is relevant for example for the study of Nichols
  algebras, can be obtained if $\rer a$ satisfies additional properties.

  Let $\cC =\cC (I,A,(\rfl _i)_{i\in I},(\Cm ^a)_{a\in A})$ be a Cartan
  scheme. For all $a\in A$ let $R^a\subset \ndZ ^I$, and define
  $m_{i,j}^a= |R^a \cap (\ndN_0 \al _i + \ndN_0 \al _j)|$ for all $i,j\in
  I$ and $a\in A$. We say that
  \[ \rsC = \rsC (\cC , (R^a)_{a\in A}) \]
  is a \textit{root system of type} $\cC $, if it satisfies the following
  axioms.
  \begin{enumerate}
    \item[(R1)]
      $R^a=R^a_+\cup - R^a_+$, where $R^a_+=R^a\cap \ndN_0^I$, for all
      $a\in A$.
    \item[(R2)]
      $R^a\cap \ndZ\al _i=\{\al _i,-\al _i\}$ for all $i\in I$, $a\in A$.
    \item[(R3)]
      $\s _i^a(R^a) = R^{\rfl _i(a)}$ for all $i\in I$, $a\in A$.
    \item[(R4)]
      If $i,j\in I$ and $a\in A$ such that $i\not=j$ and $m_{i,j}^a$ is
      finite, then
      $(\rfl _i\rfl _j)^{m_{i,j}^a}(a)=a$.
  \end{enumerate}
  The axioms (R2) and (R3) are always fulfilled for $\rsC \re $.
  The root system $\rsC $ is called \textit{finite} if for all $a\in A$ the
  set $R^a$ is finite. By \cite[Prop.\,2.12]{a-CH09a},
  if $\rsC $ is a finite root system
  of type $\cC $, then $\rsC =\rsC \re $, and hence $\rsC \re $ is a root
  system of type $\cC $ in that case.

In \cite[Def.\,4.3]{a-CH09a} the concept of an \textit{irreducible}
root system of type
$\cC $ was defined. By \cite[Prop.\,4.6]{a-CH09a}, if $\cC $ is a
Cartan scheme and $\rsC $ is a finite root system of type $\cC $, then $\rsC $
is irreducible if and only if
for all $a\in A$
the generalized Cartan matrix $C^a$ is indecomposable.
If $\cC $ is also connected, then it suffices
to require that there exists $a\in A$ such that $C^a$ is indecomposable.

  Let $\cC =\cC (I,A,(\rfl _i)_{i\in I},(\Cm ^a)_{a\in A})$ be a Cartan
  scheme.
  Let $\Gamma $ be a nondirected graph,
  such that the vertices of $\Gamma $ correspond to the elements of $A$.
  Assume that for all $i\in I$ and $a\in A$ with $\rfl _i(a)\not=a$
  there is precisely one edge between the vertices $a$ and $\rfl _i(a)$
  with label $i$,
  and all edges of $\Gamma $ are given in this way.
  The graph $\Gamma $ is called the \textit{object change diagram} of $\cC $.
\begin{figure}
\caption{The object change diagram of a Cartan scheme of rank three (nr.~15 in Table 1)}
\label{fig:14posroots}
\includegraphics[width=6cm,bb=0 10 580 580]{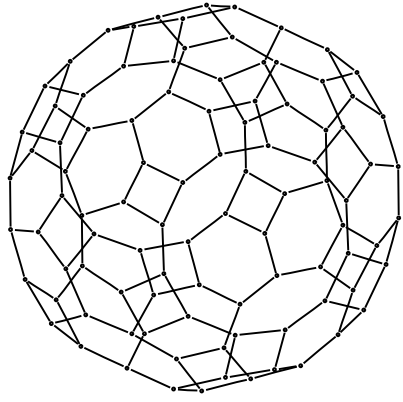}
\end{figure}

  In the rest of this section
  let $\cC=\cC (I,A,(\rfl _i)_{i\in I}, (C^a)_{a\in A})$ be a Cartan scheme
  such that $\rsC \re (\cC )$ is a finite root system.
  For brevity we will write $R^a$ instead of $\rer a$ for all $a\in A$.
  We say that a subgroup $H\subset \ndZ ^I$
  is a \textit{hyperplane} if $\ndZ ^I/H\cong \ndZ $. Then $\rk \,H=\#I-1$
  is the rank of $H$. Sometimes we will identify $\ndZ ^I$ with its image under
  the canonical embedding $\ndZ ^I\to \ndQ \otimes _\ndZ \ndZ ^I\cong \ndQ ^I$.

\begin{lemma}
  \label{le:hyperplane}
  Let $a\in A$ and let $H\subset \ndZ ^I$ be a hyperplane.
  Suppose that $H$ contains $\rk \,H$ linearly independent elements of $R^a$.
  Let $\mathfrak{n}_H$ be a normal vector of $H$ in $\ndQ ^I$ with respect
  to a scalar product $(\cdot ,\cdot )$ on $\ndQ ^I$.
  If $(\mathfrak{n}_H,\al )\ge 0$ for all $\al \in R^a_+$, then $H$ contains
  $\rk \,H$ simple roots,
  and all roots contained in $H$ are linear combinations
  of these simple roots.
\end{lemma}

\begin{proof}
  The assumptions imply that any positive root in $H$ is a linear
  combination of simple roots contained in $H$. Since $R^a=R^a_+\cup -R^a_+$,
  this implies the claim.
\end{proof}

\begin{lemma}
  Let $a\in A$ and let $H\subset \ndZ ^I$ be a hyperplane.
  Suppose that $H$ contains $\rk \,H$ linearly independent elements of $R^a$.
  Then there exist $b\in A$ and $w\in \Hom (a,b)$ such that $w(H)$ contains
  $\rk \,H$ simple roots.
  \label{le:hyperplane2}
\end{lemma}

\begin{proof}
  Let $(\cdot ,\cdot )$ be a scalar product on $\ndQ ^I$. Choose a normal
  vector $\mathfrak{n}_H$ of $H$ in $\ndQ ^I$ with respect to $(\cdot ,\cdot )$.
  Let $m=\# \{\al \in R^a_+\,|\,(\mathfrak{n}_H,\al )<0\}$. Since $\rsC \re
  (\cC )$
  is finite, $m$ is a nonnegative integer. We proceed by induction on $m$.
  If $m=0$, then $H$ contains $\rk \,H$ simple roots
  by Lemma~\ref{le:hyperplane}.
  Otherwise let $j\in I$ with $(\mathfrak{n}_H,\al _j)<0$.
  Let $a'=\rfl _j(a)$ and $H'=\s_j^a(H)$.
  Then $\s_j^a(\mathfrak{n}_H)$ is a normal vector of $H'$ with respect to the
  scalar product $(\cdot ,\cdot )'=
  (\s_j^{\rfl _j(a)}(\cdot ),\s_j^{\rfl _j(a)}(\cdot ))$.
  Since $\s_j^a:R^a_+\setminus \{\al _j\}\to R^{a'}_+\setminus \{\al _j\}$ is a
  bijection and
  $\s_j^a(\al _j)=-\al _j$, we conclude that
  \begin{align*}
    \# \{\beta \in R^{a'}_+\,|\,(\s^a_j(\mathfrak{n}_H),\beta )'<0\}=
    \# \{\al \in R^a_+\,|\,(\mathfrak{n}_H,\al )<0\}-1.
  \end{align*}
  By induction hypothesis there exists $b\in A$ and $w'\in \Hom (a',b)$ such
  that $w'(H')$ contains $\rk \,H'=\rk \,H$ simple roots.
  Then the claim of the lemma holds for $w=w'\s_j^a$.
\end{proof}

The following ``volume'' functions will be useful for our analysis.
Let $k\in \ndN $ with $k\le \#I$.
By the Smith normal form there is a unique left
$\GL (\ndZ ^I)$-invariant right $\GL (\ndZ ^k)$-invariant function
$\Vol _k:(\ndZ ^I)^k\to \ndZ $ such that
\begin{align}
  \Vol _k(a_1\al _1,\dots,a_k\al _k)=|a_1\cdots a_k| \quad
  \text{for all $a_1,\dots,a_k\in \ndZ $,}
\end{align}
where $|\cdot |$ denotes absolute value.
In particular, if $k=1$ and $\beta \in \ndZ ^I\setminus \{0\}$, then $\Vol
_1(\beta )$ is the largest integer $v$ such that $\beta =v\beta '$ for some
$\beta '\in \ndZ ^I$. Further, if $k=\#I$ and
$\beta _1,\dots,\beta _k\in \ndZ ^I$,
then $\Vol _k(\beta _1,\dots,\beta _k)$ is the absolute
value of the determinant of the matrix with columns $\beta _1,\dots,\beta _k$.

Let $a\in A$, $k\in \{1,2,\dots,\#I\}$,
and let $\beta _1,\dots,\beta _k\in R^a$
be linearly independent elements.
We write $V^a(\beta _1,\dots,\beta _k)$ for the unique maximal
subgroup $V\subseteq \ndZ ^I$ of rank $k$ which contains
$\beta _1,\dots,\beta _k$.
Then $\ndZ ^I/V^a(\beta _1,\dots,\beta _k)$ is free.
In particular, $V^a(\beta _1,\dots,\beta _{\#I})=\ndZ ^I$ for all $a\in A$ and
any linearly independent subset $\{\beta _1,\dots,\beta _{\#I}\}$ of $R^a$.

\begin{defin}
  \label{de:base}
Let $W\subseteq \ndZ ^I$ be a cofree subgroup (that is, $\ndZ ^I/W$ is free)
of rank $k$.
We say that $\{\beta _1,\dots,\beta _k\}$ is a \textit{base for $W$
at $a$}, if $\beta _i\in W$ for all $i\in \{1,\dots,k\}$ and
$W\cap R^a\subseteq \sum _{i=1}^k\ndN _0\beta _i\cup
-\sum _{i=1}^k\ndN _0\beta _i$.
\end{defin}

Now we discuss the relationship of linearly independent roots in a root
system. Recall that $\cC $ is a Cartan scheme such that $\rsC \re (\cC )$
is a finite root system of type $\cC $.

\begin{theor} \label{th:genposk}
  Let $a\in A$, $k\in \{1,\dots,\#I\}$,
  and let $\beta _1,\dots,\beta _k\in R^a$ be
  linearly independent roots.
  Then there exist $b\in A$, $w\in \Hom(a,b)$,
  and a permutation $\tau $ of $I$ such that
  \[ w(\beta _i)\in \lspan _\ndZ \{\al _{\tau (1)},
  \ldots ,\al _{\tau (i)}\} \cap R^b_+\]
  for all $i\in \{1,\dots,k\}$.
\end{theor}

\begin{proof}
  Let $r=\#I$.
  Since $R^a$ contains $r$ simple roots, any linearly independent subset of
  $R^a$ can be enlarged to a linearly independent subset of $r$ elements.
  Hence it suffices to prove the theorem for $k=r$.
  We proceed by induction on $r$.
  If $r=1$, then the claim holds.
  
  Assume that $r>1$.
  Lemma~\ref{le:hyperplane2} with
  $H=V^a(\beta _1,\dots,\beta _{r-1})$
  tells that there exist $d\in A$ and $v\in \Hom (a,d)$ such that
  $v(H)$ is spanned by simple roots.
  By multiplying $v$ from the left with
  the longest element of $\Wg (\cC )$ in the case that
  $v(\beta _r)\in -\ndN _0^I$,
  we may even assume that $v(\beta _r)\in \ndN _0^I$.
  Now let $J$ be the subset of $I$ such that $\#J=r-1$ and
  $\al _i\in v(H)$ for all $i\in J$. Consider the restriction
  $\rsC \re (\cC )|_{J}$
  of $\rsC \re (\cC )$ to the index set $J$, see \cite[Def.~4.1]{a-CH09a}.
  Since $v(\beta _i)\in H$ for all $i\in \{1,\dots,r-1\}$,
  induction hypothesis provides us with $b\in A$, $u\in \Hom (d,b)$,
  and a permutation $\tau '$ of $J$
  such that $u$ is a product of simple reflections $\s_i$, where $i\in J$,
  and
  \[ uv(\beta _n)\in 
  \lspan _\ndZ \{\al _{\tau '(j_1)}, \ldots ,\al _{\tau '(j_n)}\}
  \cap R^b_+\]
  for all $n\in \{1,2,\dots,r-1\}$, where $J=\{j_1,\dots,j_{r-1}\}$.
  Since $v(\beta _r)\notin v(H)$ and $v(\beta _r)\in \ndN _0^I$,
  the $i$-th entry of
  $v(\beta _r)$, where $i\in I\setminus J$, is positive. This entry does not
  change if we apply $u$. Therefore
  $uv(\beta _r)\in \ndN _0^I$,
  and hence the theorem holds with $w=uv\in \Hom (a,b)$ and with
  $\tau $ the unique permutation with $\tau (n)=\tau '(j_n)$
  for all $n\in \{1,\dots,r-1\}$.
\end{proof}

\begin{corol}\label{simple_rkk}
  Let $a\in A$, $k\in \{1,\dots,\#I\}$,
  and let $\beta _1,\dots,\beta _k\in R^a$
  be linearly independent elements. Then $\{\beta _1,\dots,\beta _k\}$ is a
  base for $V^a(\beta _1,\dots,\beta _k)$ at $a$ if and only if
  there exist $b\in A$, $w\in \Hom(a,b)$,
  and a permutation $\tau $ of $I$ such that
  $w(\beta _i)=\al _{\tau (i)}$ for all $i\in \{1,\dots,k\}$.
  In this case $\Vol _k(\beta _1,\dots,\beta _k)=1$.
\end{corol}

\begin{proof} The if part of the claim holds by definition of a base and by
  the axioms for root systems.

  Let $b,w$ and $\tau $ be as in Theorem~\ref{th:genposk}. Let $i\in
  \{1,\dots,k\}$.
  The elements $w(\beta _1),\dots,w(\beta _i)$ are linearly
  independent and are contained in $V^b(\al _{\tau (1)}, \dots ,
  \al _{\tau (i)})$. Thus
  $\al _{\tau (i)}$ is a rational linear combination of
  $w(\beta _1),\dots,w(\beta _i)$. Now by assumption,
  $\{w(\beta _1),\dots,w(\beta _k)\}$ is a base for
  $V^b(w(\beta _1),\dots,w(\beta _k))$ at $b$. Hence $\al _{\tau
  (i)}$ is a linear combination of the positive roots
  $w(\beta _1),\dots,w(\beta _i)$ with nonnegative integer coefficients.
  This is possible only if $\{w(\beta _1),\dots,w(\beta _i)\}$ contains
  $\al _{\tau (i)}$. By
  induction on $i$ we obtain that $\al _{\tau (i)}=w(\beta _i)$.
\end{proof}

In the special case $k=\#I$ the above corollary tells that the bases of $\ndZ
^I$ at an object $a\in A$ are precisely those subsets which
can be obtained as the image, up to a permutation, of
the standard basis of $\ndZ ^I$
under the action of an element of $\Wg (\cC )$.

In \cite{p-CH09a} the notion of an $\cF $-sequence was given, and it was used
to explain the structure of root systems of rank two.
Consider on $\ndN _0^2$ the total ordering $\le _\ndQ $,
where $(a_1,a_2)\le _\ndQ (b_1,b_2)$
if and only if $a_1 b_2\le a_2 b_1$. A finite sequence $(v_1,\dots ,v_n)$ of
vectors in $\ndN _0^2$ is an $\cF $-sequence if and only if $v_1<_\ndQ v_2
<_\ndQ \cdots <_\ndQ v_n$ and one of the following holds.
\begin{itemize}
  \item $n=2$, $v_1=(0,1)$, and $v_2=(1,0)$.
  \item $n>2$ and there exists $i\in \{2,3,\dots,n-1\}$
such that $v_i=v_{i-1}+v_{i+1}$ and $(v_1,\dots,v_{i-1}.v_{i+1},\dots,v_n)$
is an $\cF $-sequence.
\end{itemize}
In particular, any $\cF $-sequence of length $\ge 3$ contains $(1,1)$.

\begin{propo} \label{pr:R=Fseq} \cite[Prop.\,3.7]{p-CH09a}
  Let $\cC $ be a Cartan scheme of rank two.
  Assume that $\rsC \re (\cC )$ is a finite root system.
  Then for any $a\in A$
  the set $R^a_+$ ordered by $\le _\QQ$ is an $\cF $-sequence.
\end{propo}

\begin{propo} \label{pr:sumoftwo} \cite[Cor.~3.8]{p-CH09a}
  Let $\cC $ be a Cartan scheme of rank two.
  Assume that $\rsC \re (\cC )$ is a finite root system.
  Let $a\in A$ and let $\beta \in R^a_+$. Then either $\beta $ is simple or it
  is the sum of two positive roots.
\end{propo}

\begin{corol} \label{co:r2conv}
  Let $a\in A$, $n\in \ndN $,
  and let $\al ,\beta \in R^a$ such that $\beta -n\al \in R^a$.
  Assume that $\{\al ,\beta -n\al \}$ is a base for $V^a(\al ,\beta )$ at $a$.
  Then $\beta -k\al \in R^a$ for all $k\in \{1,2,\dots ,n\}$.
\end{corol}

\begin{proof}
  By Corollary~\ref{simple_rkk} there exist $b\in A$, $w\in \Hom (a,b)$,
  and $i,j\in I$ such that $w(\al )=\al _i$, $w(\beta -n\al )=\al _j$.
  Then $n\al _i+\al _j=w(\beta )\in R^b_+$.
  Hence $(n-k)\al _i+\al _j\in R^b$
  for all $k\in \{1,2,\dots,n\}$ by Proposition~\ref{pr:sumoftwo} and (R2).
  This yields the claim of the corollary.
\end{proof}

\begin{corol} \label{co:cij}
  Let $a\in A$, $k\in \ndZ $, and $i,j\in I$ such that $i\not=j$.
  Then $\al _j+k\al _i\in R^a$ if and only if $0\le k\le -c^a_{i j}$,
\end{corol}

\begin{proof}
  Axiom (R1) tells that $\al _j+k\al _i\notin R^a$ if $k<0$.
  Since $c^{\rfl _i(a)}_{i j}=c^a_{i j}$ by (C2), Axiom (R3) gives that $\al
  _j-c^a_{i j}\al _i=\sigma _i^{\rfl _i(a)}(\al _j)\in R^a$ and that
  $\al _j+k\al _i\notin R^a$ if $k>-c^a_{i j}$. Finally, if $0<k<-c^a_{i j}$
  then $\al _j+k\al _i\in R^a$ by Corollary~\ref{co:r2conv}
  for $\al =\al _i$, $\beta =\al _j-c^a_{i j}\al _i$, and $n=-c^a_{i j}$. 
\end{proof}

Proposition~\ref{pr:sumoftwo} implies another important fact.

\begin{theor}\label{root_is_sum}
  Let $\cC $ be a Cartan scheme. Assume that $\rsC \re (\cC )$ is a finite
  root system of type $\cC $.  Let $a\in A$ and $\al \in R^a_+$.
  Then either $\al $ is simple, or it is the sum of two positive roots.
\end{theor}

\begin{proof}
Assume that $\al $ is not simple.
Let $i\in I$, $b\in A$, and $w\in \Hom (b,a)$ such that $\al =w(\al _i)$.
Then $\ell(w)>0$.
We may assume that for all $j\in I$, $b'\in A$, and $w'\in \Hom (b',a)$
with $w'(\alpha_j)=\al $ we have $\ell(w')\ge\ell(w)$.
Since $w(\alpha_i)\in \ndN _0^I$, we obtain that $\ell(w\s_i)>\ell(w)$
\cite[Cor.~3]{a-HeckYam08}.
Therefore, there is a $j\in I\setminus \{i\}$ with $\ell(w\s_j)<\ell(w)$.
Let $w=w_1w_2$ such that $\ell(w)=\ell(w_1)+\ell(w_2)$,
$\ell(w_1)$ minimal and $w_2=\ldots \s_i\s_j\s_i\s_j^b$.
Assume that $w_2=\s_i\cdots \s_i\s_j^b$ --- the case
$w_2=\s_j\cdots \s_i \s_j^b$
can be treated similarly.
The length of $w_1$ is minimal,
thus $\ell(w_1\s_j)>\ell(w_1)$, and $\ell(w)=\ell(w_1)+\ell(w_2)$ yields
that $\ell(w_1\s_i)>\ell(w_1)$.
Using once more \cite[Cor.~3]{a-HeckYam08} we conclude that
\begin{align} \label{eq:twopos}
  w_1(\alpha_i)\in \ndN _0^I,\quad w_1(\alpha_j)\in \ndN _0^I.
\end{align}
Let $\beta=w_2(\alpha_i)$. Then $\beta \in \NN_0 \alpha_i+\NN_0 \alpha_j$,
since $\ell (w_2\s_i)>\ell (w_2)$. Moreover, $\beta $ is not
simple. Indeed, $\alpha=w(\alpha_i)=w_1(\beta)$, so $\beta$ is not simple, since
$\ell(w_1)<\ell(w)$ and $\ell(w)$ was chosen of minimal length.
By Proposition~\ref{pr:sumoftwo} we conclude that
$\beta$ is the sum of two positive roots
$\beta_1$, $\beta_2\in \ndN _0\al _i+\ndN _0\al _j$.
It remains to check that $w_1(\beta_1)$, $w_1(\beta_2)$
are positive. But this follows from \eqref{eq:twopos}.
\end{proof}

\section{Obstructions for Weyl groupoids of rank three}\label{rk3_obst}

In this section we analyze the structure of finite Weyl groupoids of rank
three.
Let $\cC $ be a Cartan scheme of rank
three, and assume that $\rsC \re (\cC )$
is a finite irreducible root system of type
$\cC $. In this case a hyperplane in $\ndZ ^I$ is the same as a cofree
subgroup of rank two, which will be called a \textit{plane} in the sequel.
For simplicity we will take $I=\{1,2,3\}$, and we write $R^a$ for the set of
positive real roots at $a\in A$.

Recall the definition of the functions $\Vol _k$,
where $k\in \{1,2,3\}$, from the previous section.
As noted, for three elements $\al ,\beta,\gamma\in\ZZ^3$ we have
$\Vol _3(\alpha,\beta,\gamma )=1$ if and only if $\{\alpha,\beta,\gamma\}$
is a basis of $\ndZ ^3$.
Also, we will heavily use the notion of a base, see Definition~\ref{de:base}.

\begin{lemma}\label{rootmultiple}
Let $a\in A$ and $\alpha,\beta \in R^a$.
Assume that
$\al \not=\pm \beta $ and that
$\{\al ,\beta\}$
is not a base for $V^a(\al ,\beta )$ at $a$.
Then there exist $k,l\in \ndN $ and $\delta \in R^a$
such that $\beta -k\al =l\delta $ and $\{\al ,\delta \}$ is a base for
$V^a(\al ,\beta )$ at $a$.
\end{lemma}

\begin{proof}
  The claim without the relation $k>0$ is a special case of
  Theorem~\ref{th:genposk}. The relation $\beta \not=\delta $ follows from the
  assumption that $\{\al ,\beta \}$ is not a base for $V^a(\al ,\beta )$ at
  $a$.
\end{proof}

\begin{lemma}
  Let $a\in A$ and $\alpha,\beta \in R^a$ such that $\al \not=\pm \beta $.
  Then
  $\{\al ,\beta \}$ is a base for $V^a(\al, \beta )$ if and only if
  $\Vol _2(\al ,\beta )=1$ and $\al -\beta \notin R^a$.
  \label{le:base2}
\end{lemma}

\begin{proof}
  Assume first that $\{\al ,\beta \}$ is a base for $V^a(\al ,\beta )$ at $a$.
  By Corollary~\ref{simple_rkk} we may assume that $\al $ and $\beta $ are
  simple roots. Therefore
  $\Vol _2(\al ,\beta )=1$ and $\al -\beta \notin R^a$.

  Conversely, assume that
  $\Vol _2(\al ,\beta )=1$, $\al -\beta \notin R^a$, and
  that $\{\al ,\beta \}$ is not a base for $V^a(\al ,\beta )$ at $a$.
  Let $k,l,\delta $ as in Lemma~\ref{rootmultiple}. Then
  \[ 1=\Vol _2(\al ,\beta )=\Vol _2(\al ,\beta -k\al )=l\Vol _2(\al ,\delta ).
  \]
  Hence $l=1$, and $\{\al ,\delta \}=\{\al ,\beta -k\al \}$ is
  a base for $V^a(\al ,\beta )$ at $a$. Then $\beta -\al \in R^a$
  by Corollary~\ref{co:r2conv} and since $k>0$. This gives
  the desired contradiction to the assumption $\al -\beta \notin R^a$.
\end{proof}

Recall that a semigroup ordering $<$ on a commutative semigroup $(S,+)$
is a total ordering
such that for all $s,t,u\in S$ with $s<t$ the relations
$s+u<t+u$ hold.
For example, the lexicographic ordering on $\ndZ ^I$
induced by any total ordering on $I$ is a semigroup ordering.

\begin{lemma}\label{posrootssemigroup}
  Let $a\in A$, and let $V\subset \ndZ ^I$ be a plane
  containing at least two positive roots of $R^a$.
  Let $<$ be a semigroup ordering on $\ndZ ^I$ such that $0<\gamma $
  for all $\gamma \in R^a_+$, and let
  $\al ,\beta $ denote the two smallest elements in $V\cap R^a_+$ with respect
  to $<$. Then $\{\al ,\beta \}$
  is a base for $V$ at $a$.
\end{lemma}

\begin{proof}
  Let $\al $ be the smallest element of $V\cap R^a_+$
  with respect to $<$, and let $\beta $ be the smallest element of
  $V\cap (R^a_+\setminus \{\al \})$. Then $V=V^a(\al, \beta )$ by (R2).
  By Lemma~\ref{rootmultiple} there exists $\delta \in V\cap
  R^a$ such that $\{\al ,\delta \}$ is a base for $V$ at $a$.
  First suppose that $\delta <0$. Let $m\in \ndN _0$ be the smallest
  integer with $\delta +(m+1)\al \notin R^a$. Then $\delta +n\al <0$
  for all $n\in \ndN _0$ with $n\le m$. Indeed, this holds for $n=0$ by
  assumption. By induction on $n$ we obtain from $\delta +n\al <0$
  and the choice of $\al $ that $\delta +n\al <-\al $, since $\delta $ and
  $\al $ are not collinear. Hence $\delta +(n+1)\al <0$.
  We conclude that $-(\delta +m\al )>0$. Moreover,
  $\{\al ,-(\delta +m\al )\}$ is a base for $V$ at $a$ by
  Lemma~\ref{le:base2} and the choice of $m$. Therefore,
  by replacing $\{\al ,\delta \}$ by $\{\al ,-(\delta +m\al )\}$,
  we may assume that $\delta >0$.
  Since $\beta >0$, we conclude that $\beta =k\al +l\delta $ for some $k,l\in
  \ndN _0$. Since $\beta $ is not a multiple of $\al $, this implies that
  $\beta =\delta $ or $\beta >\delta $. Then the choice of $\beta $ and the
  positivity of $\delta $ yield that $\delta =\beta $,
  that is, $\{\al ,\beta \}$ is a base for $V$ at $a$.
\end{proof}

\begin{lemma}
  \label{le:badroots}
  Let $k\in \ndN _{\ge 2}$, $a\in A$, $\al \in R^a_+$, and
  $\beta \in \ndZ ^I$ such that $\al $ and $\beta $ are not collinear and
  $\al +k\beta \in R^a$.
  Assume that $\Vol _2(\al ,\beta )=1$ and that $(-\ndN \al +\ndZ \beta )
  \cap \ndN _0^I=\emptyset $. Then $\beta \in R^a$
  and $\al +l\beta \in R^a$ for all $l\in \{1,2,\dots,k\}$.
\end{lemma}

\begin{proof}
  We prove the claim indirectly. Assume that $\beta \notin R^a$.
  By Lemma~\ref{posrootssemigroup} there exists a base $\{\gamma _1,\gamma
  _2\}$ for $V^a(\al ,\beta )$ at $a$ such that $\gamma _1,\gamma _2\in R^a_+$.
  The assumptions of the lemma imply that there exist
  $m_1,l_1\in \ndN _0$ and $m_2,l_2\in \ndZ $
  such that $\gamma _1=m_1\al +m_2\beta $, $\gamma _2=l_1\al +l_2\beta $.
  Since $\beta \notin R^a$, we obtain that $m_1\ge 1$ and $m_2\ge 1$.
  Therefore relations $\al ,\al +k\beta \in R^a_+$ imply that
  $\{\al ,\al +k\beta \}=\{\gamma _1,\gamma _2\}$. The latter
  is a contradiction to
  $\Vol _2(\gamma _1,\gamma _2)=1$ and $\Vol _2(\al ,\al +k\beta )=k>1$.
  Thus $\beta \in R^a$. By Lemma~\ref{rootmultiple}
  we obtain that $\{\beta ,\al -m\beta \}$ is a
  base for $V^a(\al ,\beta )$ at $a$ for some $m\in \ndN _0$. Then
  Corollary~\ref{co:r2conv} and the assumption that $\al +k\beta \in R^a$
  imply the last claim of the lemma.
\end{proof}

We say that a subset $S$ of $\ndZ ^3$ is \textit{convex},
if any rational convex
linear combination of elements of $S$ is either in $S$ or not in $\ndZ ^3$.
We start with a simple example.

\begin{lemma}
  \label{le:square}
  Let $a\in A$. Assume that $c^a_{12}=0$.
  
  (1) Let $k_1,k_2\in \ndZ $. Then
  $\al _3+k_1\al _1+k_2\al _2\in R^a$ if and only if $0\le k_1\le -c^a_{13}$
  and $0\le k_2\le -c^a_{23}$.
  
  (2) Let $\gamma \in (\al _3+\ndZ \al _1+\ndZ \al _2)\cap R^a$.
  Then $\gamma -\al _1\in R^a$ or $\gamma +\al _1\in R^a$. Similarly
  $\gamma -\al _2\in R^a$ or $\gamma +\al _2\in R^a$.
\end{lemma}

\begin{proof}
  (1) The assumption $c^a_{12}=0$ implies that $c^{\rfl _1(a)}_{23}=c^a_{23}$,
  see \cite[Lemma\,4.5]{a-CH09a}.
  Applying $\s _1^{\rfl _1(a)}$, $\s _2^{\rfl _2(a)}$, and
  $\s _1\s _2^{\rfl _2\rfl _1(a)}$ to $\al _3$ we conclude that
  $\al _3-c^a_{13}\al _1$, $\al _3-c^a_{23}\al _2$, $\al _3-c^a_{13}\al
  _1-c^a_{23}\al _2\in R^a_+$.
  Thus Lemma~\ref{le:badroots} implies that
  $\al _3+m_1\al _1+m_2\al _2\in R^a$
  for all $m_1,m_2\in \ndZ $ with $0\le m_1\le -c^a_{13}$
  and $0\le m_2\le -c^a_{23}$.
  Further, (R1) gives that
  $\al _3+k_1\al _1+k_2\al _2\notin R^a$ if $k_1<0$ or $k_2<0$.
  Applying again the simple reflections $\s _1$ and $\s _2$, a similar
  argument proves the remaining part of the claim.
  Observe that the proof does not use the fact
  that $\rsC \re (\cC )$ is irreducible.

  (2) Since $c^a_{12}=0$, the irreducibility of $\rsC \re (\cC )$ yields that
  $c^a_{13},c^a_{23}<0$ by \cite[Def.\,4.5, Prop.\,4.6]{a-CH09a}.
  Hence the claim follows from (1).
\end{proof}

\begin{propo} \label{pr:suminR}
  Let $a\in A$ and let $\gamma _1,\gamma _2,\gamma _3\in R^a$.
  Assume that $\Vol _3(\gamma _1,\gamma _2,\gamma _3)=1$
  and that $\gamma _1-\gamma _2,\gamma _1-\gamma _3\notin R^a$.
  Then $\gamma _1+\gamma _2\in R^a$ or $\gamma _1+\gamma _3\in R^a$.
\end{propo}

\begin{proof}
  Since $\gamma _1-\gamma _2\notin R^a$ and $\Vol _3(\gamma _1,\gamma
  _2,\gamma _3)=1$, Theorem~\ref{th:genposk} and Lemma~\ref{le:base2} imply
  that there exists $b\in A$, $w\in \Hom (a,b)$ and $i_1,i_2,i_3\in I$
  such that $w(\gamma _1)=\al _{i_1}$, $w(\gamma _2)=\al _{i_2}$,
  and $w(\gamma _3)=\al _{i_3}+k_1\al _{i_1}+k_2\al _{i_2}$ for some
  $k_1,k_2\in \ndN _0$.
  Assume that $\gamma _1+\gamma _2\notin R^a$. Then $c^b_{i_1i_2}=0$.
  Since $\gamma _3-\gamma _1\notin R^a$, Lemma~\ref{le:square}(2)
  with $\gamma =w(\gamma _3)$ gives that $\gamma _3+\gamma _1\in R^a$.
  This proves the claim.
\end{proof}

\begin{lemma} \label{le:root_diffs1}
  Assume that $R^a\cap (\ndN _0\al _1+\ndN _0\al _2)$ contains at most $4$
  positive roots.

  (1) The set $S_3:=(\al _3+\ndZ \al _1+\ndZ \al _2)\cap R^a$ is convex.
  
  (2) Let $\gamma \in S_3$. Then $\gamma =\al _3$ or $\gamma -\al _1\in R^a$ or
  $\gamma -\al _2\in R^a$.
\end{lemma}

\begin{proof}
  Consider the roots of the form $w^{-1}(\al _3)\in R^a$,
  where $w\in \Homsfrom{a}$
  is a product of reflections $\s _1^b$, $\s _2^b$ with $b\in A$. All of these
  roots belong to $S_3$. Using
  Lemma~\ref{le:badroots} the claims of the lemma can be checked case by case,
  similarly to the proof of Lemma~\ref{le:square}.
\end{proof}

\begin{remar}
  The lemma can be proven by elementary calculations, since all nonsimple
  positive roots in $(\ndZ \al _1+\ndZ \al _2)\cap R^a$
  are of the form say $\al _1+k\al _2$, $k\in \ndN $.
  We will see in Theorem~\ref{th:class}
  that the classification of connected Cartan schemes of rank three
  admitting a finite irreducible root system has a finite set of solutions.
  Thus it is possible to check the claim of the lemma for any such Cartan
  scheme. Using computer calculations one obtains that
  the lemma holds without any restriction on the (finite) cardinality of
  $R^a\cap (\ndN _0\al _1+\ndN _0\al _2)$.
\end{remar}

\begin{lemma} \label{le:root_diffs2}
  Let $\al ,\beta ,\gamma \in R^a$ such that
  $\Vol _3(\al ,\beta ,\gamma )=1$. Assume that $\al -\beta $,
  $\beta -\gamma $, $\al -\gamma \notin R^a$ and that $\{\al ,\beta ,\gamma \}$
  is not a base for $\ndZ ^I$ at $a$. Then the following hold.

  (1) There exist $w\in \Homsfrom{a}$ and $n_1,n_2\in \ndN $ such that
  $w(\al )$, $w(\beta )$, and $w(\gamma -n_1\al -n_2\beta )$ are simple roots.

  (2) None of the vectors
  $\al -k\beta $, $\al -k\gamma $, $\beta -k\al $, $\beta -k\gamma $,
  $\gamma -k\al $, $\gamma -k\beta $, where $k\in \ndN $,
  is contained in $R^a$.

  (3) $\al +\beta $, $\al +\gamma $, $\beta +\gamma \in R^a$.

  (4) One of the sets $\{\al +2\beta ,\beta +2\gamma ,\gamma +2\al \}$ and
  $\{2\al +\beta ,2\beta +\gamma ,2\gamma +\al \}$ is contained in $R^a$, the
  other one has trivial intersection with $R^a$.

  (5) None of the vectors
  $\gamma -\al -k\beta $, $\gamma -k\al -\beta $,
  $\beta -\gamma -k\al $, $\beta -k\gamma -\al $,
  $\al -\beta -k\gamma $, $\al -k\beta -\gamma $, where $k\in \ndN _0$,
  is contained in $R^a$.

  (6) Assume that $\al +2\beta \in R^a$. Let $k\in \ndN $ such that
  $\al +k\beta \in R^a$, $\al +(k+1)\beta \notin R^a$.
  Let $\al '=\al +k\beta $, $\beta '=-\beta $, $\gamma '=\gamma +\beta $.
  Then $\Vol _3(\al ',\beta ',\gamma ')=1$, $\{\al ',\beta ',\gamma '\}$ is
  not a base for $\ndZ ^I$ at $a$, and none of $\al '-\beta '$, $\al '-\gamma '$,
  $\beta '-\gamma '$ is contained in $R^a$.

  (7) None of the vectors
  $\al +3\beta $, $\beta +3\gamma $, $\gamma +3\al $,
  $3\al +\beta $, $3\beta +\gamma $, $3\gamma +\al $ is contained in $R^a$.
  In particular, $k=2$ holds in (6).
\end{lemma}

\begin{proof}
  (1) By Theorem~\ref{th:genposk} there exist $m_1,m_2,n_1,n_2,n_3\in \ndN
  _0$, $i_1,i_2,i_3\in I$, and $w\in \Homsfrom{a}$, such that
  $w(\al )=\al _{i_1}$, $w(\beta )=m_1\al _{i_1}+m_2\al _{i_2}$, and
  $w(\gamma )=n_1\al _{i_1}+n_2\al _{i_2}+n_3\al _{i_3}$. Since $\det w\in
  \{\pm 1\}$ and $\Vol _3(\al ,\beta ,\gamma )=1$,
  this implies that $m_2=n_3=1$. Further, $\beta -\al \notin R^a$,
  and hence $w(\beta )=\al _{i_2}$ by Corollary~\ref{co:cij}.
  Since $\{\al ,\beta ,\gamma \}$ is not a base for $\ndZ ^I$ at $a$,
  we conclude that $w(\gamma )\not=\al _{i_3}$. Then
  Corollary~\ref{co:cij} and the assumptions $\gamma -\al $, $\gamma -\beta
  \notin R^a$ imply that
  $w(\gamma )\notin \al _{i_3}+\ndN _0\al _{i_1}$ and
  $w(\gamma )\notin \al _{i_3}+\ndN _0\al _{i_2}$.
  Thus the claim is proven.

  (2) By (1), $\{\al ,\beta \}$ is a base for $V^a(\al ,\beta )$ at $a$.
  Thus $\al -k\beta \notin R^a$ for all $k\in \ndN $.
  The remaining claims follow by symmetry.

  (3) Suppose that $\al +\beta \notin R^a$.
  By (1) there exist $b\in A$, $w\in \Hom(a,b)$,
  $i_1,i_2,i_3\in I$ and $n_1,n_2\in \ndN $
  such that $w(\al )=\al _{i_1}$, $w(\beta )=\al _{i_2}$, and
  $w(\gamma )=\al _{i_3}+n_1\al _{i_1}+n_2\al _{i_2}\in R^b_+$.
  By Theorem~\ref{root_is_sum} there exist
  $n'_1,n'_2\in \ndN _0$  such that $n'_1\le n_1$, $n'_2\le n_2$,
  $n'_1+n'_2<n_1+n_2$, and
  \[ \al _{i_3}+n'_1\al _{i_1}+n'_2\al _{i_2}\in R^b_+,\quad
  (n_1-n'_1)\al _{i_1}+(n_2-n'_2)\al _{i_2}\in R^b_+. \]
  Since $\al +\beta \notin R^a$, Proposition~\ref{pr:R=Fseq} yields
  that $R^b_+\cap 
  \lspan _\ndZ \{\al _{i_1},\al _{i_2}\}=\{\al _{i_1},\al _{i_2}\}$.
  Thus $\gamma -\al \in R^a$ or $\gamma -\beta \in R^a$. This is a
  contradiction to the assumption of the lemma. Hence $\al +\beta \in R^a$.
  By symmetry we obtain that $\al +\gamma $, $\beta +\gamma \in R^a$.

  (4) Suppose that $\al +2\beta $, $2\al +\beta \notin R^a$.
  By (1) the set $\{\al ,\beta \}$ is a base for $V^a(\al ,\beta )$ at $a$,
  and $\al +\beta \in R^a$ by (3). Then Proposition~\ref{pr:R=Fseq}
  implies that
  $R^a\cap \lspan _\ndZ \{\al ,\beta \}=\{\pm \al ,\pm \beta ,\pm (\al +\beta
  )\}$. Thus (1) and Lemma~\ref{le:root_diffs1}(2) give that $\gamma -\al \in
  R^a$ or $\gamma -\beta \in R^a$, a contradiction to the initial assumption
  of the lemma. Hence by symmetry each of the sets $\{\al +2\beta ,2\al
  +\beta \}$, $\{\al +2\gamma ,2\al +\gamma \}$, $\{\beta +2\gamma ,2\beta
  +\gamma \}$ contains at least one element of $R^a$.

  Assume now that $\gamma +2\al $, $\gamma +2\beta \in R^a$. By changing the
  object via (1) we may assume that $\al $, $\beta $, and
  $\gamma -n_1\al -n_2\beta $
  are simple roots for some $n_1,n_2\in \ndN $. Then Lemma~\ref{le:badroots}
  applies to $\gamma +2\al \in R^a_+$ and $\beta -\al $, and tells that $\beta
  -\al \in R^a$. This gives a contradiction.

  By the previous two paragraphs we conclude that if $\gamma +2\al \in R^a$,
  then $\gamma +2\beta \notin R^a$, and hence $\beta +2\gamma \in R^a$.
  Similarly, we also obtain that $\al +2\beta \in R^a$. By symmetry this
  implies (4).

  (5) By symmetry it suffices to prove that
  $\gamma -(\al +k\beta )\notin R^a$ for all $k\in \ndN _0$. For $k=0$ the
  claim holds by assumption.

  First we prove that $\gamma -(\al +2\beta )\notin R^a$.
  By (3) we know that $\gamma +\al $, $\al +\beta \in R^a$, and
  $\gamma -\beta \notin R^a$ by assumption.
  Since $\Vol _2(\gamma +\al ,\al +\beta )=1$, Lemma~\ref{le:base2} gives that
  $\{\gamma +\al , \al +\beta \}$ is a base for
  $V^a(\gamma +\al ,\al +\beta )$ at $a$.
  Since $\gamma -(\al +2\beta )=(\gamma +\al )
  -2(\al +\beta )$, we conclude that
  $\gamma -(\al +2\beta )\notin R^a$.

  Now let $k\in \ndN $. Assume that $\gamma -(\al +k\beta )\in R^a$ and
  that $k$ is minimal with this property. Let $\al '=-\al $, $\beta '=-\beta $,
  $\gamma '=\gamma -(\al +k\beta )$. Then $\al ',\beta ',\gamma '\in R^a$
  with $\Vol _3(\al ',\beta ',\gamma ')=1$. Moreover, $\al '-\beta '\notin
  R^a$ by assumption, $\al '-\gamma '=-(\gamma -k\beta )\notin R^a$ by (2),
  and $\beta '-\gamma '=-(\gamma -\al -(k-1)\beta )\notin R^a$
  by the minimality of $k$.
  Further, $\{\al ',\beta ',\gamma '\}$ is not a base for $R^a$, since
  $\gamma =\gamma '-\al '-k\beta '$. Hence Claim (3) holds for $\al
  ',\beta ',\gamma '$. In particular,
  \[ \gamma '+\beta '=\gamma -(\al +(k+1)\beta )\in R^a. \]
  This and the previous paragraph imply that $k\ge 3$.

  We distinguish two cases depending on the parity of $k$.
  First assume that $k$ is even. Let $\al '=\gamma +\al $ and
  $\beta '=-(\al +k/2\beta )$. Then $\Vol _2(\al ',\beta ')=1$ and
  $\al '+2\beta '=\gamma -(\al +k\beta )\in R^a$.
  Lemma~\ref{le:badroots} applied to $\al ',\beta '$ gives that $\gamma
  -k/2\beta =\al '+\beta '\in R^a$, which contradicts (2).

  Finally, the case of odd $k$ can be excluded similarly by considering
  $V^a(\gamma +\al ,\gamma -(\al +(k+1)\beta ))$.

  (6) We get $\Vol _3(\al ',\beta ',\gamma ')=1$
  since $\Vol _3(\al ,\beta ,\gamma )=1$ and
  $\Vol _3$ is invariant under the right action of $\GL (\ndZ ^3)$.
  Further, $\beta '-\gamma '=-(2\beta +\gamma )\notin R^a$ by (4), and
  $\al '-\gamma '\notin R^a$ by (5). Finally, $(\al ',\beta ',\gamma ')$ is
  not a base for $\ndZ ^I$ at $a$, since
  $R^a\ni \gamma -n_1\al -n_2\beta =\gamma '-n_1\al '+(1+n_2-kn_1)\beta '$,
  where $n_1,n_2\in \ndN $ are as in (1).

  (7) We prove that $\gamma +3\al \notin R^a$. The rest follows by symmetry.
  If $2\al +\beta \in R^a$, then $\gamma +2\al \notin R^a$ by (4), and hence
  $\gamma +3\al \notin R^a$. Otherwise
  $\al +2\beta ,\gamma +2\al \in R^a$ by (4).
  Let $k$, $\al '$, $\beta '$, $\gamma '$ be as in (6). Then (6) and (3) give
  that
  $R^a\ni \gamma '+\al '=\gamma +\al +(k+1)\beta $. Since $\gamma +\al \in
  R^a$, Lemma~\ref{le:badroots} implies that
  $\gamma +\al +2\beta \in R^a$. Let $w$ be as in (1).
  If $\gamma +3\al \in R^a$, then Lemma~\ref{le:badroots}
  for the vectors $w(\gamma +\al +2\beta )$ and $w(\al -\beta )$ implies that
  $w(\al -\beta )\in R^a$, a contradiction. Thus $\gamma +3\al \notin R^a$.
\end{proof}

Recall that $\cC$ is a Cartan scheme of rank three and $\rsC \re (\cC )$ is
a finite irreducible root system of type $\cC$.

\begin{theor}\label{root_diffs}
Let $a\in A$ and $\al ,\beta,\gamma\in R^a$.
If $\Vol _3(\al ,\beta,\gamma)=1$
and none of $\alpha-\beta$, $\alpha-\gamma$, $\beta-\gamma$ are contained in
$R^a$, then $\{\al ,\beta,\gamma \}$ is a base for $\ndZ ^I$ at $a$.
\end{theor}

\begin{proof}
  Assume to the contrary that $\{\al ,\beta ,\gamma \}$ is not a base for
  $\ndZ ^I$ at $a$.
  Exchanging $\al $ and $\beta $ if necessary, by
  Lemma~\ref{le:root_diffs2}(4)
  we may assume that $\al +2\beta \in R^a$.
  By Lemma~\ref{le:root_diffs2}(6),(7) the triple $(\al +2\beta
  ,-\beta , \gamma +\beta )$ satisfies the assumptions of
  Lemma~\ref{le:root_diffs2}, and $(\al +2\beta )+2(-\beta )=\al \in R^a$.
  Hence $2\al +3\beta =2(\al +2\beta )+(-\beta )\notin R^a$ by
  Lemma~\ref{le:root_diffs2}(4). Thus $V^a(\al ,\beta )\cap R^a=\{\pm \al ,
  \pm (\al +\beta ),\pm (\al +2\beta ), \pm \beta \}$ by
  Proposition~\ref{pr:R=Fseq}, and hence, using
  Lemma~\ref{le:root_diffs2}(1), we obtain from Lemma~\ref{le:root_diffs1}(2)
  that $\gamma -\al \in R^a$ or $\gamma -\beta \in R^a$. This is a
  contradiction to our initial assumption, and hence $\{\al ,\beta ,\gamma \}$
  is a base for $\ndZ ^I$ at $a$.
\end{proof}

\begin{corol}\label{convex_diff2}
Let $a\in A$ and $\gamma _1,\gamma _2,\al \in R^a$.
Assume that $\{\gamma _1,\gamma _2\}$ is a base for $V^a(\gamma _1,\gamma _2)$
at $a$ and that $\Vol _3(\gamma _1,\gamma _2,\al )=1$. Then
either $\{\gamma _1,\gamma _2,\al \}$ is a base for $\ndZ ^I$ at $a$,
or one of $\al -\gamma _1$, $\al-\gamma _2$ is contained in $R^a$.
\end{corol}

For the proof of Theorem~\ref{th:class} we need a bound for the entries of the
Cartan matrices of $\cC $. To get this bound we use the following.

\begin{lemma} \label{le:someroots}
  Let $a\in A$. 

  (1) At most one of $c^a_{12}$, $c^a_{13}$, $c^a_{23}$ is zero.

  (2) $\al _1+\al _2+\al _3\in R^a$.

  (3) Let $k\in \ndZ $.
  Then $k\al _1+\al _2+\al _3\in R^a$ if and only if $k_1\le k\le k_2$, where
  \begin{align*}
    k_1=
    \begin{cases}
      0 & \text{if $c^a_{23}<0$,}\\
      1 & \text{if $c^a_{23}=0$,}
    \end{cases}
    \quad
    k_2=
    \begin{cases}
      -c^a_{12}-c^a_{13} & \text{if $c^{\rfl _1(a)}_{23}<0$,}\\
      -c^a_{12}-c^a_{13}-1 & \text{if $c^{\rfl _1(a)}_{23}=0$.}
    \end{cases}
  \end{align*}

  (4) We have
  $2\al _1+\al _2+\al _3\in R^a$ if and only if either
  $c^a_{12}+c^a_{13}\le -3$ or $c^a_{12}+c^a_{13}=-2$, $c^{\rfl
  _1(a)}_{23}<0$.

  (5) Assume that
  \begin{align}
    \#(R^a_+\cap (\ndZ \al _1+\ndZ \al _2))\ge 5.
    \label{eq:Rbig}
  \end{align}
  Then there
  exist $k\in \ndN _0$ such that $k\al _1+2\al _2+\al _3\in R^a$. Let $k_0$
  be the smallest among all such $k$. Then $k_0$ is given by the following.
  \begin{align*}
    \begin{cases}
      0 & \text{if $c^a_{23}\le -2$,}\\
      1 & \text{if $-1\le c^a_{23}\le 0$,
      $c^a_{21}+c^a_{23}\le -2$, $c^{\rfl _2(a)}_{13}<0$,}\\
      1 & \text{if $-1\le c^a_{23}\le 0$,
      $c^a_{21}+c^a_{23}\le -3$, $c^{\rfl _2(a)}_{13}=0$,}\\
      2 & \text{if $c^a_{21}=c^a_{23}=-1$, $c^{\rfl _2(a)}_{13}=0$,}\\
      2 & \text{if $c^a_{21}=-1$, $c^a_{23}=0$, $c^{\rfl _2(a)}_{13}\le -2$,}\\
      3 & \text{if $c^a_{21}=-1$, $c^a_{23}=0$, $c^{\rfl _2(a)}_{13}=-1$,
      $c^{\rfl _2(a)}_{12}\le -3$,}\\
      3 & \text{if $c^a_{21}=-1$, $c^a_{23}=0$, $c^{\rfl _2(a)}_{13}=-1$,
      $c^{\rfl _2(a)}_{12}=-2$, $c^{\rfl _1\rfl _2(a)}_{23}<0$,}\\
      4 & \text{otherwise.}
    \end{cases}
  \end{align*}
  Further, if $c^a_{13}=0$ then $k_0\le 2$.
\end{lemma}

\begin{proof}
  We may assume that $\cC $ is connected. Then, since $\rsC \re (\cC )$ is irreducible,
  Claim (1) holds by \cite[Def.\,4.5, Prop.\,4.6]{a-CH09a}.

  (2) The claim is invariant under permutation of $I$.
  Thus by (1) we may assume that $c^a_{23}\not=0$. Hence $\al _2+\al _3\in
  R^a$. Assume first that $c^a_{13}=0$. Then $c^{\rfl _1(a)}_{13}=0$ by (C2),
  $c^{\rfl _1(a)}_{23}\not=0$ by (1), and $\al _2+\al _3\in R^{\rfl _1(a)}_+$.
  Hence $\s^{\rfl _1(a)}_1(\al _2+\al _3)=-c^a_{12}\al _1+\al _2+\al _3\in
  R^a$. Therefore (2) holds by Lemma~\ref{le:badroots}
  for $\al =\al _2+\al _3$ and $\beta =\al _1$.

  Assume now that $c^a_{13}\not=0$. By symmetry and the previous paragraph we
  may also assume that $c^a_{12},c^a_{23}\not=0$. Let $b=\rfl _1(a)$.
  If $c^b_{23}=0$ then $\al _1+\al _2+\al _3\in R^b$ by the previous
  paragraph. Then
  \[ R^a\ni \s^b_1(\al _1+\al _2+\al _3)
     =(-c^a_{12}-c^a_{13}-1)\al _1+\al _2+\al _3,
  \]
  and the coefficient of $\al _1$ is positive.
  Further, $\al _2+\al _3\in R^a$, and hence (2) holds in this case by
  Lemma~\ref{le:badroots}. Finally, if $c^b_{23}\not=0$, then $\al _2+\al
  _3\in R^b_+$, and hence $(-c^a_{12}-c^a_{13})\al _1+\al _2+\al _3\in R^a$.
  Since $-c^a_{12}-c^a_{13}>0$, (2) follows again from
  Lemma~\ref{le:badroots}.

  (3) If $c^a_{23}<0$, then $\al _2+\al _3\in R^a$ and
  $-k\al _1+\al _2+\al _3\notin R^a$ for all $k\in \ndN $.
  If $c^a_{23}=0$, then $\al _1+\al _2+\al _3\in R^a$ by (2), and
  $-k\al _1+\al _2+\al _3\notin R^a$ for all $k\in \ndN _0$.
  Applying the same argument to $R^{\rfl _1(a)}$ and using the reflection
  $\s^{\rfl _1(a)}_1$ and Lemma~\ref{le:badroots} gives the claim.

  (4) This follows immediately from (3).

  (5) The first case follows from Corollary~\ref{co:cij} and the second and
  third cases are obtained from (4) by interchanging the elements $1$ and $2$
  of $I$. We also obtain that if $k_0$ exists then
  $k_0\ge 2$ in all other cases. By \eqref{eq:Rbig} and
  Proposition~\ref{pr:R=Fseq} we conclude that $\al _1+\al _2\in R^a$. Then
  $c^a_{21}<0$ by Corollary~\ref{co:cij}, and hence
  we are left with calculating $k_0$ if
  $-1\le c^a_{23}\le 0$, $c^a_{21}+c^a_{23}=-2$, $c^{\rfl _2(a)}_{13}=0$, or
  $c^a_{21}=-1$, $c^a_{23}=0$. By (1), if $c^{\rfl _2(a)}_{13}=0$ then
  $c^{\rfl _2(a)}_{23}\not=0$, and hence $c^a_{23}<0$ by (C2).
  Thus we have to consider the elements $k\al _1+2\al _2+\al _3$, where
  $k\ge 2$, under the assumption that
  \begin{align}
    c^a_{21}=c^a_{23}=-1, \, c^{\rfl _2(a)}_{13}=0 \quad \text{or}\quad
    c^a_{21}=-1,\, c^a_{23}=0.
    \label{eq:ccond1}
  \end{align}
  Since $c^a_{21}=-1$, Condition~\eqref{eq:Rbig} gives that
  \[ c^{\rfl _2(a)}_{12}\le -2, \]
  see \cite[Lemma\,4.8]{a-CH09a}. Further,
  the first set of equations in \eqref{eq:ccond1}
  implies that $c^{\rfl _1\rfl _2(a)}_{13}=0$,
  and hence $c^{\rfl _1\rfl _2(a)}_{23}<0$ by (1). 
  Since $\s _2^a(2\al _1+2\al _2+\al _3)=2\al _1+\al _3-c^a_{23}\al
  _2$, the first set of equations in \eqref{eq:ccond1} and (4) imply that
  $k_0=2$. Similarly, Corollary~\ref{co:cij} tells that
  $k_0=2$ under the second set of conditions in \eqref{eq:ccond1}
  if and only if $c^{\rfl _2(a)}_{13}\le -2$.

  It remains to consider the situation for
  \begin{align}
    c^a_{21}=-1,\,c^a_{23}=0,\,c^{\rfl _2(a)}_{13}=-1.
    \label{eq:ccond2}
  \end{align}
  Indeed, equation $c^a_{23}=0$ implies that
  $c^{\rfl _2(a)}_{23}=0$ by (C2), and hence
  $c^{\rfl _2(a)}_{13}<0$ by (1),
  Assuming \eqref{eq:ccond2} we obtain that
  $\s _2^a(3\al _1+2\al _2+\al _3)=3\al _1+\al _2+\al _3$,
  and hence (3) implies that $k_0=3$ if and only if the corresponding
  conditions in (5) are valid.

  The rest follows by looking at $\sigma _1 \sigma _2^a(4\al _1+2\al _2+\al
  _3)$ and is left to the reader. The last claim holds since $c^a_{13}=0$
  implies that $c^a_{23}\not=0$ by (1).
  The assumption $\#(R^a_+\cap (\ndZ \al _1+\ndZ \al _2))\ge 5$
  is needed to exclude the case
  $c^a_{21}=-1$, $c^{\rfl _2(a)}_{12}=-2$, $c^{\rfl _1\rfl
  _2(a)}_{21}=-1$, where 
  $R^a_+\cap (\ndZ \al _1+\ndZ \al _2)=\{\al _2,\al _1+\al _2,
  2\al _1+\al _2,\al _1\}$, by using Proposition~\ref{pr:R=Fseq} and
  Corollary~\ref{co:cij}, see also the proof of
  \cite[Lemma\,4.8]{a-CH09a}.
\end{proof}

\begin{theor}\label{cartan_6}
Let $\cC$ be a Cartan scheme of rank three. Assume that $\rsC \re $
is a finite irreducible root system
of type $\cC$. Then all entries of the Cartan matrices of $\cC$ are greater or
equal to $-7$.
\end{theor}

\begin{proof}
  It can be assumed that $\cC $ is connected.
  We prove the theorem indirectly. To do so we may assume that $a\in A$ such
  that $c^a_{12}\le -8$. Then Proposition~\ref{pr:R=Fseq} implies that
  $\# (R^a_+\cap (\ndZ \al _1+\ndZ \al _2))\ge 5$.
  By Lemma~\ref{le:someroots} there exists $k_0\in \{0,1,2,3,4\}$ such that
  $\al :=k_0\al _1+2\al _2+\al _3\in R^a_+$ and
  $\al -\al _1\notin R^a$.
  By Lemma~\ref{le:base2} and the choice of $k_0$ the set $\{\al ,\al _1\}$
  is a base for $V^a(\al ,\al _1)$ at $a$. Corollary~\ref{simple_rkk}
  implies that there exists a root
  $\gamma \in R^a$ such that $\{\al ,\al _1,\gamma \}$ is a base for $\ndZ ^I$
  at $a$. Let $d\in A$, $w\in \Hom (a,d)$, and $i_1,i_2,i_3\in I$
  such that $w(\al )=\al _{i_1}$, $w(\al _1)=\al _{i_2}$,
  $w(\gamma )=\al _{i_3}$.

  Let $b=\rfl _1(a)$. Again by
  Lemma~\ref{le:someroots} there exists $k_1\in \{0,1,2,3,4\}$ such that
  $\beta :=k_1\al _1+2\al _2+\al _3\in R^b_+$.
  Thus
  \[ R^a_+\ni \s_1^b(\beta )=(-k_1-2c^a_{12}-c^a_{13})\al _1+2\al _2+\al
  _3. \]
  Further,
  \[ -k_1-2c^a_{12}-c^a_{13}-k_0>-c^a_{12} \]
  since $k_0\le 2$ if $c^a_{13}=0$. Hence $\al _{i_1}+(1-c^a_{12})\al _{i_2}
  \in R^d$, that is, $c^d_{i_2 i_1}<c^a_{1 2}\le -8$.
  We conclude that there exists
  no lower bound for the entries of the Cartan matrices of $\cC $,
  which is a contradiction to the finiteness of $\rsC \re (\cC )$.
  This proves the theorem.
\end{proof}

\begin{remar}
  The bound in the theorem is not sharp. After completing the classification
  one can check that the entries of the Cartan matrices of $\cC $ are always
  at least $-6$. The entry $-6$ appears for example in the Cartan scheme
  corresponding to the root system with number $53$, see
  Corollary~\ref{co:cij}.
\end{remar}

\begin{propo}\label{Euler_char}
Let $\cC $ be an irreducible connected simply connected Cartan scheme of rank three.
Assume that $\rsC \re (\cC )$ is a finite root system of type
$\cC $.
Let $e$ be the number of vertices, $k$ the number of edges, and $f$ the
number of ($2$-dimensional)
faces of the object change diagram of $\cC $. Then $e-k+f=2$.
\end{propo}

\begin{proof}
  Vertices of the object change diagram correspond to elements of $A$.
  Since $\cC $ is connected and $\rsC \re (\cC )$ is finite, the set $A$ is
  finite.
  Consider the equivalence relation on $I\times A$, where $(i,a)$ is
  equivalent to $(j,b)$ for some $i,j\in I$ and $a,b\in A$,
  if and only if $i=j$ and $b\in \{a,\rfl _i(a)\}$. (This is also known as the
  pushout of $I\times A$ along the bijections $\id :I\times A\to I\times A$
  and $\rfl :I\times A\to I\times A$, $(i,a)\mapsto (i,\rfl _i(a))$.)
  Since $\cC $ is simply
  connected, $\rfl _i(a)\not=a$ for all $i\in I$ and $a\in A$. Then
  edges of the object change diagram
  correspond to equivalence classes in $I\times A$.
  Faces of the object change diagram can be defined as equivalence classes of
  triples $(i,j,a)\in I\times I\times A\setminus
  \{(i,i,a)\,|\,i\in I,a\in A\}$, where $(i,j,a)$ and $(i',j',b)$ are
  equivalent for some $i,j,i',j'\in I$ and $a,b\in A$ if and only if
  $\{i,j\}=\{i',j'\}$ and $b\in \{(\rfl _j\rfl _i)^m(a),
  \rfl _i(\rfl _j\rfl _i)^m(a)\,|\,m\in \ndN _0\}$. Since $\cC $ is simply
  connected, (R4) implies that the face corresponding to
  a triple $(i,j,a)$ is a polygon with $2m_{i,j}^a$ vertices.

For each face choose a
triangulation by non-intersecting diagonals.
Let $d$ be the total number of diagonals
arising this way.
Now consider the following two-dimensional simplicial complex $C$:
The $0$-simplices are the objects. The $1$-simplices are the edges and the
chosen diagonals of the faces of the
object change diagram.
The $2$-simplices are the $f+d$ triangles.
Clearly, each edge is contained in precisely two triangles.
By \cite[Ch.\,III, (3.3), 2,3]{b-tomDieck91}
the geometric realization $X$ of $C$ is a closed $2$-dimensional
surface without boundary. The space $X$ is connected and compact.

Any two morphisms of $\Wg (\cC )$ with same source and target are
equal because $\cC $ is simply connected.
By \cite[Thm.\,2.6]{a-CH09a} this equality follows from
the Coxeter relations. A Coxeter relation means for the object
change diagram that for the corresponding face and vertex the two paths
along the sides of the face towards the opposite vertex yield the same
morphism. Hence diagonals can be interpreted as representatives of paths in a
face between two vertices, and then all loops in $C$
become homotopic to the trivial loop.
Hence $X$ is simply connected
and therefore homeomorphic to a two-dimensional sphere by
\cite[Ch.\,III, Satz 6.9]{b-tomDieck91}.
Its Euler characteristic is $2=e-(k+d)+(f+d)=e-k+f$.
\end{proof}

\begin{remar} \label{re:planesandfaces}
Assume that $\cC $ is connected and simply connected, and let $a\in A$.
Then any pair of opposite $2$-dimensional faces of the object change diagram
can be interpreed as a plane in $\ndZ ^I$ containing at least
two positive roots $\al ,\beta \in R^a_+$.
Indeed, let $b\in A$ and $i_1,i_2\in I$ with $i_1\not=i_2$. Since $\cC $ is
connected and simply connected, there exists a unique $w\in \Hom (a,b)$. Then
$V^a(w^{-1}(\al _{i_1}),w^{-1}(\al _{i_2}))$ is a plane in $\ndZ ^I$
containing at least two positive roots. One can easily check that this plane
is independent of the choice of the representative of the face determined by
$(i_1,i_2,b)\in I\times I\times A$. Further, let $w_0\in \Hom (b,d)$, where
$d\in A$, be the longest element in $\Homsfrom b$. Let $j_1,j_2\in I$ such
that $w_0(\al _{i_n})=-\al _{j_n}$ for $n=1,2$. Then $(j_1,j_2,d)$ determines
the plane
\[ V^a( (w_0w)^{-1}(\al _{j_1}),(w_0w)^{-1}(\al _{j_2}))=
V^a(w^{-1}(\al _{i_1}),w^{-1}(\al _{i_2})). \]
This way we attached to any pair of ($2$-dimensional) opposite faces
of the object change diagram a plane containing at least two positive roots.
\begin{figure}
\caption{The object change diagram of the last root system of rank three}
\label{fig:37posroots}
\includegraphics[width=9cm,bb=0 50 520 550]{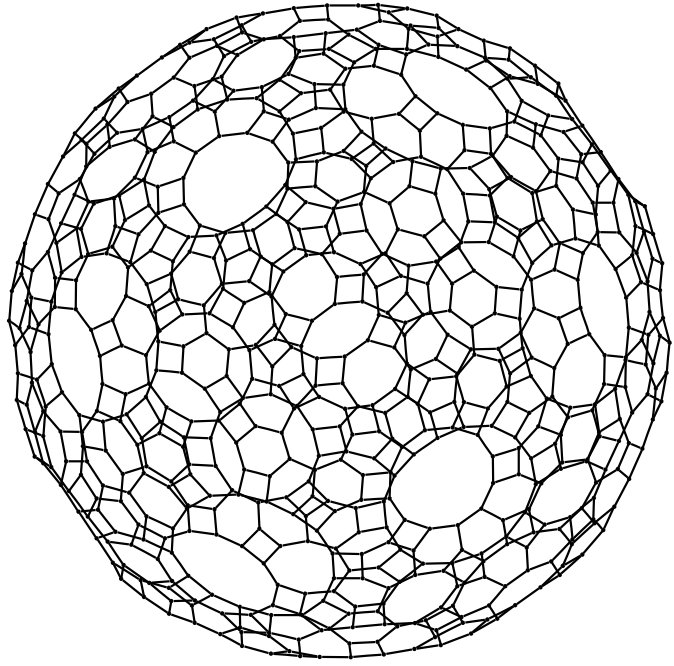}
\end{figure}
Let $<$ be a semigroup ordering on $\ndZ ^I$ such that $0<\gamma $ for
all $\gamma \in R^a_+$. Let $\al ,\beta \in R^a_+$ with $\al \not=\beta $, and
assume that $\al $ and $\beta $ are the smallest elements in $R^a_+\cap
V^a(\al ,\beta )$ with respect to $<$. Then $\{\al ,\beta \}$ is a base for
$V^a(\al ,\beta )$ at $a$ by Lemma~\ref{posrootssemigroup}. By
Corollary~\ref{simple_rkk} there exists $b\in A$ and $w\in \Hom (a,b)$ such
that $w(\al ),w(\beta )\in R^b_+$ are simple roots.
Hence any plane in $\ndZ ^I$ containing at least two elements of $R^a_+$
can be obtained by the construction in the previous paragraph.

It remains to show that different pairs of opposite faces give rise to
different planes. This follows from the fact that for any $b\in A$ and
$i_1,i_2\in I$ with $i_1\not=i_2$ the conditions
\[ d\in A,\ u\in \Hom (b,d),\ j_1,j_2\in I,\
u(\al _{i_1})=\al _{j_1},\ u(\al _{i_2})=\al _{j_2}
\]
have precisely two solutions: $u=\id _b$ on the one side, and
$u=w_0w_{i_1i_2}$ on the other side, where
$w_{i_1i_2}=\cdots \s_{i_1}\s_{i_2}\s_{i_1}\id _b\in \Homsfrom b$
is the longest product of reflections $\s _{i_1}$, $\s _{i_2}$,
and $w_0$ is an appropriate longest element of $\Wg (\cC )$.
The latter follows from the fact that $u$ has to map the base $\{\al
_{i_1},\al _{i_2},\al _{i_3}\}$ for $\ndZ ^I$ at $b$,
where $I=\{i_1,i_2,i_3\}$, to another base, and any base consisting of two
simple roots can be extended precisely in two ways to a base of $\ndZ ^I$: by
adding the third simple root or by adding a uniquely determined negative root.

It follows from the construction and by \cite[Lemma~6.4]{a-CH09b}
that the faces corresponding to a plane
$V^a(\al ,\beta )$, where $\al ,\beta \in R^a_+$ with $\al \not=\beta $, have
as many edges as the cardinality of $V^a(\al ,\beta )\cap R^a$ (or twice the
cardinality of $V^a(\al ,\beta )\cap R^a_+$).
\end{remar}

\begin{theor}\label{sum_rank2}
Let $\cC$ be a connected simply connected
Cartan scheme of rank three. Assume that
$\rsC \re (\cC )$ is a finite irreducible root system of type $\cC$.
Let $a\in A$ and let $M$ be the set of planes containing at least two elements
of $R^a_+$. Then
\[ \sum_{V\in M} \#(V\cap R^a_+) = 3(\# M-1). \]
\end{theor}

\begin{proof}
  Let $e,k,f$ be as in Proposition~\ref{Euler_char}.
  Then $\#M=f/2$ by Remark~\ref{re:planesandfaces}.
  For any vertex $b\in A$ there are three edges starting at $b$,
  and any edge is bounded by two vertices. Hence counting vertices in two
  different ways one obtains that $3e=2k$.
  Proposition~\ref{Euler_char} gives that $e-k+f=2$. Hence
  $2k = 3e = 3(2-f+k)$, that is, $k=3f-6$.

  Any plane $V$ corresponds to a face which is a polygon consisting of
  $2\# (V\cap R^a_+)$ edges, see Remark~\ref{re:planesandfaces}.
  Summing up the edges
  twice over all planes (that is summing up over all faces of the
  object change diagram), each edge is counted twice.
  Hence
  \[ 2 \sum_{V\in M} 2\#(V\cap R^a_+) = 2k = 2(3f-6), \]
  which is the formula claimed in the theorem.
\end{proof}

\begin{corol}\label{ex_square}
Let $\cC$ be a connected simply connected Cartan scheme of rank three.
Assume that $\rsC \re (\cC )$ is a finite irreducible root system of type $\cC$.
Then there exists an object $a\in A$ and
$\alpha,\beta,\gamma\in R^a_+$ such that
$\{\alpha,\beta,\gamma\}=\{\al_1,\al_2,\al_3\}$ and
\begin{equation}\label{square_hexagon}
\#(V^a(\alpha,\beta)\cap R^a_+)=2, \quad
\#(V^a(\alpha,\gamma)\cap R^a_+)=3.
\end{equation}
Further $\alpha+\gamma, \beta +\gamma , \alpha+\beta+\gamma\in R^a_+$.
\end{corol}
\begin{proof}
Let $M$ be as in Thm.~\ref{sum_rank2}. Let $a$ be any object and
assume $\#(V\cap R^a_+)>2$ for all $V\in M$, then
$\sum_{V\in M} \#(V\cap R^a_+) \ge 3\# M$ contradicting Thm.~\ref{sum_rank2}.
Hence for all objects $a$ there exists a plane $V$ with $\#(V\cap R^a_+)=2$.
Now consider the object change diagram and count the number of faces:
Let $2q_i$ be the number of faces with $2i$ edges. Then Thm.~\ref{sum_rank2}
translates to
\begin{equation}\label{thm_trans}
\sum_{i\ge 2} i q_i = -3+3\sum_{i\ge 2} q_i.
\end{equation}
Assume that there exists no object adjacent to a square and a hexagon.
Since $\rsC \re (\cC )$ is irreducible, no two squares have a common edge, see
Lemma~\ref{le:someroots}(1).
Look at the edges ending in vertices of squares, and count each edge once for
both polygons adjacent to it. One checks that there are
at least twice as many
edges adjacent to a polygon with at least $8$ vertices as edges
of squares.
This gives that
\[ \sum_{i\ge 4} 2i \cdot 2q_i \ge 2\cdot 4\cdot 2q_2. \]
By Equation~\eqref{thm_trans} we then have
$-3+3\sum_{i\ge 2}q_i\ge 4q_2+2q_2+3q_3$, that is, $q_2 < \sum_{i\ge 4}q_i$.
But then in average each face has more than $6$ edges which
contradicts Thm.~\ref{sum_rank2}.
Hence there is an object $a$ such that there exist
$\alpha,\beta,\gamma\in R^a_+$
as above satisfying Equation~\eqref{square_hexagon}.
We have
$\alpha+\gamma, \beta +\gamma , \alpha+\beta+\gamma\in R^a_+$
by Lemma~\ref{le:someroots}(1),(2) and Corollary~\ref{co:cij}.
\end{proof}

\section{The classification}
\label{sec:class}

In this section we explain the classification of connected simply connected
Cartan schemes of rank three such that $\rsC \re (\cC )$ is a finite
irreducible root system.
We formulate the main result in Theorem~\ref{th:class}.
The proof of Theorem~\ref{th:class} is performed using computer calculations
based on results of the previous sections.
Our algorithm described below is sufficiently powerful: The
implementation in $C$ terminates within a few hours on a usual computer.
Removing any of the theorems,
the calculations would take at least several weeks.

\begin{theor}
  (1) Let $\cC$ be a connected Cartan scheme of rank three with $I=\{1,2,3\}$.
  Assume that $\rsC \re (\cC )$ is a finite
  irreducible root system of type $\cC$. Then there exists an object $a\in A$
  and a linear map $\tau \in \Aut (\ndZ ^I)$
  such that
  $\tau (\al _i)\in \{\al _1,\al _2,\al _3\}$ for all $i\in I$ and
  $\tau (R^a_+)$ is one of the sets listed in
  Appendix~\ref{ap:rs}. Moreover, $\tau (R^a_+)$ with this property is
  uniquely determined.

  (2) Let $R$ be one of the $55$ subsets of $\ndZ ^3$ appearing in
  Appendix~\ref{ap:rs}. There exists up to equivalence a unique
  connected simply connected
  Cartan scheme $\cC (I,A,(\rfl _i)_{i\in I},(C^a)_{a\in A})$
  such that $R\cup -R$ is the set of real roots $R^a$ in an object $a\in A$. 
  Moreover $\rsC \re (\cC )$ is a finite irreducible root system of type $\cC
  $.
  \label{th:class}
\end{theor}

Let $<$ be the lexicographic ordering on $\ndZ ^3$ such
that $\al_3<\al_2<\al_1$.
Then $\al >0$ for any $\al \in \ndN _0^3\setminus \{0\}$.

Let $\cC $ be a connected Cartan scheme with $I=\{1,2,3\}$.
Assume that $\rsC \re (\cC )$ is a
finite irreducible root system of type $\cC $.
Let $a\in A$.
By Theorem~\ref{root_is_sum} we may construct $R^a_+$ inductively
by starting with $R^a_+=\{\al _3,\al _2,\al _1\}$, and
appending in each step a sum of a pair of positive roots which is greater
than all roots in $R^a_+$ we already have.
During this process, we keep track of all planes containing at least two
positive roots, and the positive roots on them.
Lemma~\ref{posrootssemigroup} implies that for any known root $\al $ and new
root $\beta $ either $V^a(\al ,\beta )$ contains no other known positive
roots,
or $\beta $ is not part of the unique base for $V^a(\al ,\beta )$ at $a$
consisting of positive roots.
In the first case the roots $\al ,\beta $ generate a new plane. It can happen
that $\Vol _2(\al ,\beta )>1$, and then $\{\al ,\beta \}$ is not a base
for $V^a(\al ,\beta )$ at $a$,
but we don't care about that.
In the second case the known roots in $V^a(\al ,\beta )\cap R^a_+$
together with $\beta $ have to form an $\cF$-sequence by
Proposition~\ref{pr:R=Fseq}.

Sometimes, by some theorem (see the details below)
we know that it is not possible to add more
positive roots to a plane. Then we can mark it as ``finished''.

Remark that to obtain a finite number of root systems as output,
we have to ensure that we compute only irreducible systems since there are
infinitely many inequivalent reducible root systems of rank two.
Hence starting with $\{\al _3,\al _2,\al _1\}$ will not work.
However, by Corollary~\ref{ex_square}, starting
with $\{\al _3,\al _2,\al_2+\al_3,\al _1,\al_1+\al_2,\al_1+\al_2+\al_3\}$
will still yield at least one root system for each desired Cartan scheme
(notice that any roots one would want to add are lexicographically greater).

In this section, we will call {\it root system fragment} (or {\it rsf}) the
following set of data associated to a set of positive roots $R$
in construction:
\begin{itemize}
\item normal vectors for the planes with at least two positive roots
\item labels of positive roots on these planes
\item Cartan entries corresponding to the root systems of the planes
\item an array of flags for finished planes
\item the sum $s_R$ of $\#(V\cap R)$ over all planes $V$ with at least two positive
  roots, see Theorem \ref{sum_rank2}
\item for each root $r\in R$ the list of planes it belongs to.
\end{itemize}
These data can be obtained directly from $R$, but the calculation is faster
if we continuously update them.

We divide the algorithm into three parts.
The main part is Algorithm~4.4, see below.

The first part updates a root system fragment to a new root and uses
Theorems \ref{root_diffs} and \ref{cartan_6} to possibly refuse doing so:

\algo{AppendRoot}{$\alpha$,$B$,$\tilde B$,$\hat\alpha$}
{Append a root to an rsf}
{a root $\alpha$, an rsf $B$, an empty rsf $\tilde B$, a root $\hat\alpha$.}
{$\begin{cases}
  0: & \mbox{if } \alpha \mbox{ may be appended, new rsf is then in } \tilde B, \\
  1: & \mbox{if } \alpha \mbox{ may not be appended}, \\
  2: & \mbox{if $\alpha \in R^a_+$ implies the existence of
  $\beta \in R^a_+$}\\
  & \mbox{with $\hat\alpha<\beta <\alpha $.}
\end{cases}$}
{
\item Let $r$ be the number of planes containing at least two elements of $R$.
  For documentation purposes let $V_1,\dots,V_r$ denote these planes.
  For any $i\in \{1,\dots,r\}$ let $v_i$ be a normal vector for $V_i$,
  and let $R_i$ be the $\cF$-sequence of $V_i\cap R$. Set
  $i \leftarrow 1$, $g \leftarrow 1$, $c\leftarrow [\:]$,
  $p \leftarrow [\:]$, $d\leftarrow\{\:\}$. During the algorithm $c$ will be
  an ordered subset of $\{1,\dots,r\}$, $p$ a corresponding list of
  ``positions'', and $d$ a subset of $R$.
\item \label{A1_2} If $i\le r$ and $g\ne 0$, then compute the scalar product
  $g:=(\alpha , v_i)$. (Then $g=\det (\al ,\gamma _1,\gamma _2)
  =\pm \Vol _3(\al ,\gamma _1,\gamma _2)$, where $\{\gamma _1,\gamma _2\}$
  is the basis of $V_i$
  consisting of positive roots.) Otherwise go to Step \ref{A1_6}.
\item \label{A1_3} If $g=0$ then do the following:
        If $V_i$ is not finished yet, then check if $\al $ extends $R_i$ to a
        new $\cF $-sequence.
        If yes, add the roots of $R_i$ to $d$,
        append $i$ to $c$, append the position of the insertion of $\al $ in
        $R_i$ to $p$,
        let $g \leftarrow 1$, and go to Step~5.
\item If $g^2=1$, then use Corollary~\ref{convex_diff2}:
      Let $\gamma_1$ and $\gamma_2$ be the beginning and the end of the $\cF
      $-sequence $R_i$, respectively. (Then $\{\gamma _1,\gamma _2\}$ is a
      base for $V_i$ at $a$). Let
      $\delta_1 \leftarrow \alpha - \gamma_1$,
      $\delta_2 \leftarrow \alpha - \gamma_2$.
      If $\delta_1,\delta _2\notin R$, then return $1$ if $\delta
      _1,\delta _2 \le \hat \al $ and return $2$ otherwise.
\item Set $i \leftarrow i+1$ and go to Step \ref{A1_2}.
\item \label{A1_6}
    If there is no objection appending $\alpha$ so far, i.e.~$g \ne 0$, then
    copy $B$ to $\tilde B$ and include $\alpha$ into $\tilde B$:
    use $c,p$ to extend existing $\cF $-sequences, and use (the complement of)
    $d$ to create new planes.
    Finally, apply Theorem \ref{cartan_6}:
    If there is a Cartan entry lesser than $-7$
    then return 1, else return 0.
    If $g=0$ then return 2.
}

The second part looks for small roots which we must include
in any case. The function is based on Proposition~\ref{pr:suminR}.
This is a strong restriction during the process.

\algo{RequiredRoot}{$R$,$B$,$\hat\alpha$}
{Find a smallest required root under the assumption that all roots $\le \hat
\al $ are known}
{$R$ a set of roots, $B$ an rsf for $R$, $\hat \alpha $ a root.}
{$\begin{cases}
  0 & \mbox{if we cannot determine such a root}, \\
  1,\varepsilon & \mbox{if we have found a small missing root $\varepsilon $
  with $\varepsilon >\hat \alpha $},\\
  2 & \mbox{if the given configuration is impossible}.
\end{cases}$}
{
\item Initialize the return value $f \leftarrow 0$.
\item \label{A2_0}
  We use the same notation as in Algo.~4.2, step 1.
  For all $\gamma _1$ in $R$ and all
  $(j,k)\in \{1,\dots,r\}\times \{1,\dots,r\}$
  such that $j\not=k$, $\gamma _1\in R_j\cap R_k$, and both $R_j,R_k$
  contain two elements,
  let $\gamma _2,\gamma _3\in R$ such that $R_j=\{\gamma _1,\gamma _2\}$,
  $R_k=\{\gamma _1,\gamma _3\}$.
  If $\Vol _3(\gamma _1,\gamma _2,\gamma _3) = 1$,
  then do Steps \ref{A2_a} to \ref{A2_b}.
\item \label{A2_a}
$\xi_2 \leftarrow \gamma_1+\gamma_2$,
$\xi_3 \leftarrow \gamma_1+\gamma_3$.
\item If $\hat\alpha \ge \xi_2$:
      If $\hat\alpha \ge \xi_3$ or plane $V_k$ is already finished,
      then return 2.
        If $f=0$ or $\varepsilon > \xi_3$, then
          $\varepsilon \leftarrow \xi_3$, $f\leftarrow 1$.
          Go to Step \ref{A2_0} and continue loop.
\item If $\hat\alpha \ge \xi_3$:
        If plane $V_j$ is already finished, then return 2.
        If $f=0$ or $\varepsilon > \xi_2$, then
          $\varepsilon \leftarrow \xi_2$, $f\leftarrow 1$.
\item \label{A2_b} Go to Step \ref{A2_0} and continue loop.
\item Return $f,\varepsilon$.
}

Finally, we resursively add roots to a set, update the rsf and
include required roots:

\algo{CompleteRootSystem}{$R$,$B$,$\hat\alpha$,$u$,$\beta$}
{\label{mainalg}Collects potential new roots, appends them and calls itself again}
{$R$ a set of roots, $B$ an rsf for $R$,
$\hat\alpha$ a lower bound for new roots, $u$ a flag,
$\beta$ a vector which is necessarily a root if $u=$ True.}
{Root systems containing $R$.}
{\label{A3} 
\item Check Theorem \ref{sum_rank2}: If $s_R = 3(r-1)$, where $r$ is the
  number of planes containing at least two positive roots, then
  output $R$ (and continue). We have found a potential root system.
\item If we have no required root yet, i.e.~$u=$ False, then\\
$f,\varepsilon:=$RequiredRoot$(R,B,\hat\al )$.
If $f=1$, then we have found a required root; we call
CompleteRootSystem($R,B,\hat\alpha, True, \varepsilon$) and terminate.
If $f=2$, then terminate.
\item Potential new roots will be collected in $Y\leftarrow \{\:\}$;
$\tilde B$ will be the new rsf.
\item For all planes $V_i$ of $B$ which are not finished,
 do Steps \ref{A3_a} to \ref{A3_b}.
\item \label{A3_a} $\nu \leftarrow 0$.
\item For $\zeta$ in the set of roots that may be added to the
  plane $V_i$ such that $\zeta> \hat\alpha$, do the following:
\begin{itemize}
\item set $\nu \leftarrow \nu+1$.
\item If $\zeta \notin Y$, then $Y \leftarrow Y \cup \{\zeta\}$.
If moreover $u=$ False or $\beta > \zeta$, then
\begin{itemize}
\item $y \leftarrow$ AppendRoot($\zeta,B,\tilde B,\hat\alpha$);
\item if $y = 0$ then CompleteRootSystem($R\cup\{\zeta\},\tilde B,\zeta ,
  u, \beta$).
\item if $y = 1$ then $\nu \leftarrow \nu-1$.
\end{itemize}
\end{itemize}
\item \label{A3_b} If $\nu = 0$, then mark $V_i$ as finished in $\tilde B$.
\item if $u =$ True and AppendRoot($\beta,B,\tilde B,\hat\alpha$) = 0,
  then call
  \\CompleteRootSystem($R\cup\{\beta\},\tilde B,\beta, \textrm{False},
\beta$).\\
Terminate the function call.
}

Note that we only used necessary conditions for root systems,
so after the computation we still need to check which of the sets
are indeed root systems.
A short program in {\sc Magma} confirms that Algorithm 4.4 yields
only root systems, for instance using this algorithm:

\algo{RootSystemsForAllObjects}{$R$}
{Returns the root systems for all objects if $R=R^a_+$ determines a Cartan
scheme $\cC $ such that $\rsC \re (\cC )$ is an irreducible root system}
{$R$ the set of positive roots at one object.}
{the set of root systems at all objects, or $\{\}$ if $R$ does
not yield a Cartan scheme as desired.}
{\label{A4}
\item $N \leftarrow [R]$, $M \leftarrow \{\}$.
\item While $|N| > 0$, do steps \ref{begwhile} to \ref{endwhile}.
\item Let $F$ be the last element of $N$. Remove $F$ from $N$ and include it to $M$.\label{begwhile}
\item Let $C$ be the Cartan matrix of $F$. Compute the three simple reflections given by $C$.
\item For each simple reflection $s$, do:\label{endwhile}
\begin{itemize}
\item Compute $G:=\{s(v)\mid v\in F\}$. If an element of $G$ has positive and negative coefficients, then return $\{\}$. Otherwise mutliply the negative roots of $G$ by $-1$.
\item If $G\notin M$, then append $G$ to $N$.
\end{itemize}
\item Return $M$.}

We list all $55$ root systems in Appendix~\ref{ap:rs}.
It is also interesting to
summarize some of the invariants, which is done in Table~1.
Let ${\mathcal O}=\{R^a \mid a \in A\}$ denote the set
of different root systems. By identifying objects with the same root system
one obtains a quotient Cartan scheme of the simply connected Cartan scheme of
the classification. In the fifth column we give the automorphism group of one
(equivalently, any) object of this quotient.
The last column gives the multiplicities of planes;
for example $3^7$ means that there are $7$ different planes containing
precisely $3$ positive roots.

\begin{center}
\begin{longtable}{|l l l l l l|}
\hline
Nr. & $|R_+^a|$ & $|{\mathcal O}|$ & $|A|$ & $\Hom(a)$ & planes \\
\hline
\endhead
\hline
\endfoot
$1$ & $6$ & $1$ & $24$ & $A_3$ & $2^{3}, 3^{4}, $\\
$2$ & $7$ & $4$ & $32$ & $A_1\times A_1\times A_1$ & $2^{3}, 3^{6}, $\\
$3$ & $8$ & $5$ & $40$ & $B_2$ & $2^{4}, 3^{6}, 4^{1}, $\\
$4$ & $9$ & $1$ & $48$ & $B_3$ & $2^{6}, 3^{4}, 4^{3}, $\\
$5$ & $9$ & $1$ & $48$ & $B_3$ & $2^{6}, 3^{4}, 4^{3}, $\\
$6$ & $10$ & $5$ & $60$ & $A_1\times A_2$ & $2^{6}, 3^{7}, 4^{3}, $\\
$7$ & $10$ & $10$ & $60$ & $A_2$ & $2^{6}, 3^{7}, 4^{3}, $\\
$8$ & $11$ & $9$ & $72$ & $A_1\times A_1\times A_1$ & $2^{7}, 3^{8}, 4^{4}, $\\
$9$ & $12$ & $21$ & $84$ & $A_1\times A_1$ & $2^{8}, 3^{10}, 4^{3}, 5^{1}, $\\
$10$ & $12$ & $14$ & $84$ & $A_2$ & $2^{9}, 3^{7}, 4^{6}, $\\
$11$ & $13$ & $4$ & $96$ & $G_2\times A_1$ & $2^{9}, 3^{12}, 4^{3}, 6^{1}, $\\
$12$ & $13$ & $12$ & $96$ & $A_1\times A_1\times A_1$ & $2^{10}, 3^{10}, 4^{3}, 5^{2}, $\\
$13$ & $13$ & $2$ & $96$ & $B_3$ & $2^{12}, 3^{4}, 4^{9}, $\\
$14$ & $13$ & $2$ & $96$ & $B_3$ & $2^{12}, 3^{4}, 4^{9}, $\\
$15$ & $14$ & $56$ & $112$ & $A_1$ & $2^{11}, 3^{12}, 4^{4}, 5^{2}, $\\
$16$ & $15$ & $16$ & $128$ & $A_1\times A_1\times A_1$ & $2^{13}, 3^{12}, 4^{6}, 5^{2}, $\\
$17$ & $16$ & $36$ & $144$ & $A_1\times A_1$ & $2^{14}, 3^{15}, 4^{6}, 5^{1}, 6^{1}, $\\
$18$ & $16$ & $24$ & $144$ & $A_2$ & $2^{15}, 3^{13}, 4^{6}, 5^{3}, $\\
$19$ & $17$ & $10$ & $160$ & $B_2\times A_1$ & $2^{16}, 3^{16}, 4^{7}, 6^{2}, $\\
$20$ & $17$ & $10$ & $160$ & $B_2\times A_1$ & $2^{16}, 3^{16}, 4^{7}, 6^{2}, $\\
$21$ & $17$ & $10$ & $160$ & $B_2\times A_1$ & $2^{18}, 3^{12}, 4^{7}, 5^{4}, $\\
$22$ & $18$ & $30$ & $180$ & $A_2$ & $2^{18}, 3^{18}, 4^{6}, 5^{3}, 6^{1}, $\\
$23$ & $18$ & $90$ & $180$ & $A_1$ & $2^{19}, 3^{16}, 4^{6}, 5^{5}, $\\
$24$ & $19$ & $25$ & $200$ & $A_1\times A_1\times A_1$ & $2^{20}, 3^{20}, 4^{6}, 5^{4}, 6^{1}, $\\
$25$ & $19$ & $8$ & $192$ & $G_2\times A_1$ & $2^{21}, 3^{18}, 4^{6}, 6^{4}, $\\
$26$ & $19$ & $50$ & $200$ & $A_1\times A_1$ & $2^{20}, 3^{20}, 4^{6}, 5^{4}, 6^{1}, $\\
$27$ & $19$ & $25$ & $200$ & $A_1\times A_1\times A_1$ & $2^{20}, 3^{20}, 4^{6}, 5^{4}, 6^{1}, $\\
$28$ & $19$ & $8$ & $192$ & $G_2\times A_1$ & $2^{24}, 3^{12}, 4^{6}, 5^{6}, 6^{1}, $\\
$29$ & $20$ & $27$ & $216$ & $B_2$ & $2^{20}, 3^{26}, 4^{4}, 5^{4}, 8^{1}, $\\
$30$ & $20$ & $110$ & $220$ & $A_1$ & $2^{21}, 3^{24}, 4^{6}, 5^{4}, 7^{1}, $\\
$31$ & $20$ & $110$ & $220$ & $A_1$ & $2^{23}, 3^{20}, 4^{7}, 5^{5}, 6^{1}, $\\
$32$ & $21$ & $15$ & $240$ & $B_2\times A_1$ & $2^{22}, 3^{28}, 4^{6}, 5^{4}, 8^{1}, $\\
$33$ & $21$ & $30$ & $240$ & $A_1\times A_1\times A_1$ & $2^{26}, 3^{20}, 4^{9}, 5^{4}, 6^{2}, $\\
$34$ & $21$ & $5$ & $240$ & $B_3$ & $2^{24}, 3^{24}, 4^{9}, 6^{4}, $\\
$35$ & $22$ & $44$ & $264$ & $A_2$ & $2^{27}, 3^{25}, 4^{9}, 5^{3}, 6^{3}, $\\
$36$ & $25$ & $42$ & $336$ & $A_1\times A_1\times A_1$ & $2^{33}, 3^{34}, 4^{12}, 5^{2}, 6^{3}, 8^{1}, $\\
$37$ & $25$ & $14$ & $336$ & $G_2\times A_1$ & $2^{36}, 3^{30}, 4^{9}, 5^{6}, 6^{4}, $\\
$38$ & $25$ & $28$ & $336$ & $A_1\times A_2$ & $2^{36}, 3^{30}, 4^{9}, 5^{6}, 6^{4}, $\\
$39$ & $25$ & $7$ & $336$ & $B_3$ & $2^{36}, 3^{28}, 4^{15}, 6^{6}, $\\
$40$ & $26$ & $182$ & $364$ & $A_1$ & $2^{35}, 3^{39}, 4^{10}, 5^{4}, 6^{3}, 8^{1}, $\\
$41$ & $26$ & $182$ & $364$ & $A_1$ & $2^{37}, 3^{36}, 4^{9}, 5^{6}, 6^{3}, 7^{1}, $\\
$42$ & $27$ & $49$ & $392$ & $A_1\times A_1\times A_1$ & $2^{38}, 3^{42}, 4^{9}, 5^{6}, 6^{3}, 8^{1}, $\\
$43$ & $27$ & $98$ & $392$ & $A_1\times A_1$ & $2^{39}, 3^{40}, 4^{10}, 5^{6}, 6^{2}, 7^{2}, $\\
$44$ & $27$ & $98$ & $392$ & $A_1\times A_1$ & $2^{39}, 3^{40}, 4^{10}, 5^{6}, 6^{2}, 7^{2}, $\\
$45$ & $28$ & $420$ & $420$ & $1$ & $2^{41}, 3^{44}, 4^{11}, 5^{6}, 6^{2}, 7^{1}, 8^{1}, $\\
$46$ & $28$ & $210$ & $420$ & $A_1$ & $2^{42}, 3^{42}, 4^{12}, 5^{6}, 6^{1}, 7^{3}, $\\
$47$ & $28$ & $70$ & $420$ & $A_2$ & $2^{42}, 3^{42}, 4^{12}, 5^{6}, 6^{1}, 7^{3}, $\\
$48$ & $29$ & $56$ & $448$ & $A_1\times A_1\times A_1$ & $2^{44}, 3^{46}, 4^{13}, 5^{6}, 6^{2}, 8^{2}, $\\
$49$ & $29$ & $112$ & $448$ & $A_1\times A_1$ & $2^{45}, 3^{44}, 4^{14}, 5^{6}, 6^{1}, 7^{2}, 8^{1}, $\\
$50$ & $29$ & $112$ & $448$ & $A_1\times A_1$ & $2^{45}, 3^{44}, 4^{14}, 5^{6}, 6^{1}, 7^{2}, 8^{1}, $\\
$51$ & $30$ & $238$ & $476$ & $A_1$ & $2^{49}, 3^{44}, 4^{17}, 5^{6}, 6^{1}, 7^{1}, 8^{2}, $\\
$52$ & $31$ & $21$ & $504$ & $G_2\times A_1$ & $2^{54}, 3^{42}, 4^{21}, 5^{6}, 6^{1}, 8^{3}, $\\
$53$ & $31$ & $21$ & $504$ & $G_2\times A_1$ & $2^{54}, 3^{42}, 4^{21}, 5^{6}, 6^{1}, 8^{3}, $\\
$54$ & $34$ & $102$ & $612$ & $A_2$ & $2^{60}, 3^{63}, 4^{18}, 5^{6}, 6^{4}, 8^{3}, $\\
$55$ & $37$ & $15$ & $720$ & $B_3$ & $2^{72}, 3^{72}, 4^{24}, 6^{10}, 8^{3}, $
\end{longtable}
{\small Table 1: Invariants of irreducible root systems of rank three}
\end{center}

At first sight, one is tempted to look for a formula for the number of objects
in the universal covering depending on the number of roots. There is an
obvious one: consider the coefficients of $4/((1-x)^2(1-x^4))$. However, there
are exceptions, for example nr.~29 with $20$ positive roots
and $216$ objects (instead of $220$).

Rank 3 Nichols algebras of diagonal type with finite irreducible
arithmetic root system are classified in \cite[Table 2]{a-Heck05b}. In
Table~2 we identify the Weyl groupoids of these Nichols
algebras.

\vspace{\baselineskip}
\begin{center}
\begin{tabular}{|l | l l l l l l l l l |}
\hline
row in \cite[Table 2]{a-Heck05b}
& 1 & 2 & 3 & 4 & 5 & 6 & 7 & 8 & 9\\
Weyl groupoid &
1 & 5 & 4 & 1 & 5 & 3 & 11 & 1 & 2\\
\hline
row in \cite[Table 2]{a-Heck05b}
& 10 & 11 & 12 & 13 & 14 & 15 & 16 & 17 & 18 \\
Weyl groupoid &
2 & 2 & 5 & 13 & 5 & 6 & 7 & 8 & 14\\
\hline
\end{tabular}

\vspace{5pt}
{\small Table 2: Weyl groupoids of rank three Nichols algebras of
diagonal type}
\end{center}

\begin{appendix}
  \section{Irreducible root systems of rank three}
\label{ap:rs}
We give the roots in a multiplicative notation to save
space: The word $1^x2^y3^z$ corresponds to
$x\alpha_3+y\alpha_2+z\alpha_1$.

Notice that we have chosen a ``canonical'' object for each
groupoid. Write $\pi(R^a_+)$ for the set $R^a_+$ where the coordinates
are permuted via $\pi\in S_3$. Then the set listed below is the minimum
of $\{\pi(R^a_+)\mid a\in A,\:\: \pi\in S_3\}$ with respect
to the lexicographical ordering on the sorted sequences of roots.
\vspace{11pt}

\baselineskip=10pt

\begin{tiny}
\noindent
Nr. $1$ with $6$ positive roots:\\
$1$, $2$, $3$, $12$, $13$, $123$\\
Nr. $2$ with $7$ positive roots:\\
$1$, $2$, $3$, $12$, $13$, $23$, $123$\\
Nr. $3$ with $8$ positive roots:\\
$1$, $2$, $3$, $12$, $13$, $1^{2}2$, $123$, $1^{2}23$\\
Nr. $4$ with $9$ positive roots:\\
$1$, $2$, $3$, $12$, $13$, $1^{2}2$, $123$, $1^{2}23$, $1^{2}23^{2}$\\
Nr. $5$ with $9$ positive roots:\\
$1$, $2$, $3$, $12$, $23$, $1^{2}2$, $123$, $1^{2}23$, $1^{2}2^{2}3$\\
Nr. $6$ with $10$ positive roots:\\
$1$, $2$, $3$, $12$, $13$, $1^{2}2$, $1^{2}3$, $123$, $1^{2}23$, $1^{3}23$\\
Nr. $7$ with $10$ positive roots:\\
$1$, $2$, $3$, $12$, $13$, $23$, $1^{2}2$, $123$, $1^{2}23$, $1^{2}2^{2}3$\\
Nr. $8$ with $11$ positive roots:\\
$1$, $2$, $3$, $12$, $13$, $1^{2}2$, $1^{2}3$, $123$, $1^{2}23$, $1^{3}23$, $1^{3}2^{2}3$\\
Nr. $9$ with $12$ positive roots:\\
$1$, $2$, $3$, $12$, $13$, $1^{2}2$, $123$, $1^{3}2$, $1^{2}23$, $1^{3}23$, $1^{3}2^{2}3$, $1^{4}2^{2}3$\\
Nr. $10$ with $12$ positive roots:\\
$1$, $2$, $3$, $12$, $13$, $1^{2}2$, $1^{2}3$, $123$, $1^{2}23$, $1^{3}23$, $1^{2}2^{2}3$, $1^{3}2^{2}3$\\
Nr. $11$ with $13$ positive roots:\\
$1$, $2$, $3$, $12$, $13$, $1^{2}2$, $123$, $1^{3}2$, $1^{2}23$, $1^{3}2^{2}$, $1^{3}23$, $1^{3}2^{2}3$, $1^{4}2^{2}3$\\
Nr. $12$ with $13$ positive roots:\\
$1$, $2$, $3$, $12$, $13$, $1^{2}2$, $1^{2}3$, $123$, $1^{3}2$, $1^{2}23$, $1^{3}23$, $1^{4}23$, $1^{4}2^{2}3$\\
Nr. $13$ with $13$ positive roots:\\
$1$, $2$, $3$, $12$, $13$, $1^{2}2$, $1^{2}3$, $123$, $1^{2}23$, $1^{3}23$, $1^{2}2^{2}3$, $1^{3}2^{2}3$, $1^{4}2^{2}3$\\
Nr. $14$ with $13$ positive roots:\\
$1$, $2$, $3$, $12$, $13$, $1^{2}2$, $123$, $13^{2}$, $1^{2}23$, $123^{2}$, $1^{2}23^{2}$, $1^{3}23^{2}$, $1^{3}2^{2}3^{2}$\\
Nr. $15$ with $14$ positive roots:\\
$1$, $2$, $3$, $12$, $13$, $1^{2}2$, $1^{2}3$, $123$, $1^{3}2$, $1^{2}23$, $1^{3}23$, $1^{4}23$, $1^{3}2^{2}3$, $1^{4}2^{2}3$\\
Nr. $16$ with $15$ positive roots:\\
$1$, $2$, $3$, $12$, $13$, $1^{2}2$, $1^{2}3$, $123$, $1^{3}2$, $1^{2}23$, $1^{3}23$, $1^{4}23$, $1^{3}2^{2}3$, $1^{4}2^{2}3$, $1^{5}2^{2}3$\\
Nr. $17$ with $16$ positive roots:\\
$1$, $2$, $3$, $12$, $13$, $1^{2}2$, $1^{2}3$, $123$, $1^{3}2$, $1^{2}23$, $1^{3}2^{2}$, $1^{3}23$, $1^{4}23$, $1^{3}2^{2}3$, $1^{4}2^{2}3$, $1^{5}2^{2}3$\\
Nr. $18$ with $16$ positive roots:\\
$1$, $2$, $3$, $12$, $23$, $1^{2}2$, $123$, $1^{3}2$, $1^{2}23$, $12^{2}3$, $1^{3}23$, $1^{2}2^{2}3$, $1^{3}2^{2}3$, $1^{4}2^{2}3$, $1^{4}2^{3}3$, $1^{4}2^{3}3^{2}$\\
Nr. $19$ with $17$ positive roots:\\
$1$, $2$, $3$, $12$, $13$, $1^{2}2$, $1^{2}3$, $123$, $1^{3}2$, $1^{2}23$, $1^{4}2$, $1^{3}23$, $1^{4}23$, $1^{5}23$, $1^{4}2^{2}3$, $1^{5}2^{2}3$, $1^{6}2^{2}3$\\
Nr. $20$ with $17$ positive roots:\\
$1$, $2$, $3$, $12$, $13$, $1^{2}2$, $1^{2}3$, $123$, $1^{3}2$, $1^{2}23$, $1^{3}2^{2}$, $1^{3}23$, $1^{4}23$, $1^{3}2^{2}3$, $1^{4}2^{2}3$, $1^{5}2^{2}3$, $1^{5}2^{2}3^{2}$\\
Nr. $21$ with $17$ positive roots:\\
$1$, $2$, $3$, $12$, $13$, $1^{2}2$, $123$, $1^{3}2$, $1^{2}23$, $1^{3}23$, $1^{2}2^{2}3$, $1^{3}2^{2}3$, $1^{4}2^{2}3$, $1^{5}2^{2}3$, $1^{5}2^{3}3$, $1^{5}2^{3}3^{2}$,
$1^{6}2^{3}3^{2}$\\
Nr. $22$ with $18$ positive roots:\\
$1$, $2$, $3$, $12$, $13$, $1^{2}2$, $1^{2}3$, $123$, $1^{3}2$, $1^{2}23$, $1^{3}2^{2}$, $1^{3}23$, $1^{4}23$, $1^{3}2^{2}3$, $1^{4}2^{2}3$, $1^{5}2^{2}3$, $1^{5}2^{3}3$,
$1^{6}2^{3}3$\\
Nr. $23$ with $18$ positive roots:\\
$1$, $2$, $3$, $12$, $13$, $23$, $1^{2}2$, $123$, $1^{3}2$, $1^{2}23$, $12^{2}3$, $1^{3}23$, $1^{2}2^{2}3$, $1^{3}2^{2}3$, $1^{4}2^{2}3$, $1^{3}2^{3}3$, $1^{4}2^{3}3$,
$1^{4}2^{3}3^{2}$\\
Nr. $24$ with $19$ positive roots:\\
$1$, $2$, $3$, $12$, $13$, $1^{2}2$, $123$, $1^{3}2$, $1^{2}23$, $1^{4}2$, $1^{3}23$, $1^{4}23$, $1^{3}2^{2}3$, $1^{4}2^{2}3$, $1^{5}2^{2}3$, $1^{6}2^{2}3$, $1^{6}2^{3}3$,
$1^{7}2^{3}3$, $1^{7}2^{3}3^{2}$\\
Nr. $25$ with $19$ positive roots:\\
$1$, $2$, $3$, $12$, $23$, $1^{2}2$, $123$, $1^{3}2$, $1^{2}23$, $1^{4}2$, $1^{3}23$, $1^{2}2^{2}3$, $1^{4}23$, $1^{3}2^{2}3$, $1^{4}2^{2}3$, $1^{5}2^{2}3$, $1^{6}2^{2}3$,
$1^{6}2^{3}3$, $1^{6}2^{3}3^{2}$\\
Nr. $26$ with $19$ positive roots:\\
$1$, $2$, $3$, $12$, $13$, $23$, $1^{2}2$, $12^{2}$, $123$, $1^{3}2$, $1^{2}23$, $12^{2}3$, $1^{3}23$, $1^{2}2^{2}3$, $1^{3}2^{2}3$, $1^{4}2^{2}3$, $1^{3}2^{3}3$, $1^{4}2^{3}3$,
$1^{4}2^{3}3^{2}$\\
Nr. $27$ with $19$ positive roots:\\
$1$, $2$, $3$, $12$, $13$, $23$, $1^{2}2$, $123$, $1^{3}2$, $1^{2}23$, $12^{2}3$, $1^{3}2^{2}$, $1^{3}23$, $1^{2}2^{2}3$, $1^{3}2^{2}3$, $1^{4}2^{2}3$, $1^{3}2^{3}3$,
$1^{4}2^{3}3$, $1^{4}2^{3}3^{2}$\\
Nr. $28$ with $19$ positive roots:\\
$1$, $2$, $3$, $12$, $23$, $1^{2}2$, $123$, $1^{3}2$, $1^{2}23$, $1^{3}2^{2}$, $1^{3}23$, $1^{2}2^{2}3$, $1^{3}2^{2}3$, $1^{4}2^{2}3$, $1^{3}2^{3}3$, $1^{4}2^{3}3$,
$1^{5}2^{3}3$, $1^{6}2^{3}3$, $1^{6}2^{4}3$\\
Nr. $29$ with $20$ positive roots:\\
$1$, $2$, $3$, $12$, $13$, $1^{2}2$, $123$, $1^{3}2$, $1^{2}23$, $1^{4}2$, $1^{3}2^{2}$, $1^{3}23$, $1^{4}23$, $1^{3}2^{2}3$, $1^{5}2^{2}$, $1^{4}2^{2}3$, $1^{5}2^{2}3$,
$1^{6}2^{2}3$, $1^{6}2^{3}3$, $1^{7}2^{3}3$\\
Nr. $30$ with $20$ positive roots:\\
$1$, $2$, $3$, $12$, $13$, $1^{2}2$, $123$, $1^{3}2$, $1^{2}23$, $1^{4}2$, $1^{3}2^{2}$, $1^{3}23$, $1^{4}23$, $1^{3}2^{2}3$, $1^{4}2^{2}3$, $1^{5}2^{2}3$, $1^{6}2^{2}3$,
$1^{6}2^{3}3$, $1^{7}2^{3}3$, $1^{7}2^{3}3^{2}$\\
Nr. $31$ with $20$ positive roots:\\
$1$, $2$, $3$, $12$, $13$, $1^{2}2$, $1^{2}3$, $123$, $1^{3}2$, $1^{2}23$, $1^{3}2^{2}$, $1^{3}23$, $1^{2}2^{2}3$, $1^{4}23$, $1^{3}2^{2}3$, $1^{4}2^{2}3$, $1^{5}2^{2}3$,
$1^{4}2^{3}3$, $1^{5}2^{3}3$, $1^{6}2^{3}3^{2}$\\
Nr. $32$ with $21$ positive roots:\\
$1$, $2$, $3$, $12$, $13$, $1^{2}2$, $123$, $1^{3}2$, $1^{2}23$, $1^{4}2$, $1^{3}2^{2}$, $1^{3}23$, $1^{4}23$, $1^{3}2^{2}3$, $1^{5}2^{2}$, $1^{4}2^{2}3$, $1^{5}2^{2}3$,
$1^{6}2^{2}3$, $1^{6}2^{3}3$, $1^{7}2^{3}3$, $1^{7}2^{3}3^{2}$\\
Nr. $33$ with $21$ positive roots:\\
$1$, $2$, $3$, $12$, $13$, $1^{2}2$, $1^{2}3$, $123$, $1^{3}2$, $1^{2}23$, $1^{3}2^{2}$, $1^{3}23$, $1^{2}2^{2}3$, $1^{4}23$, $1^{3}2^{2}3$, $1^{4}2^{2}3$, $1^{5}2^{2}3$,
$1^{4}2^{3}3$, $1^{5}2^{3}3$, $1^{6}2^{3}3$, $1^{6}2^{3}3^{2}$\\
Nr. $34$ with $21$ positive roots:\\
$1$, $2$, $3$, $12$, $13$, $1^{2}2$, $1^{2}3$, $123$, $1^{3}2$, $1^{2}23$, $1^{3}2^{2}$, $1^{3}23$, $1^{4}23$, $1^{3}2^{2}3$, $1^{4}2^{2}3$, $1^{5}2^{2}3$, $1^{5}2^{3}3$,
$1^{5}2^{2}3^{2}$, $1^{6}2^{3}3$, $1^{6}2^{3}3^{2}$, $1^{7}2^{3}3^{2}$\\
Nr. $35$ with $22$ positive roots:\\
$1$, $2$, $3$, $12$, $13$, $1^{2}2$, $1^{2}3$, $123$, $1^{3}2$, $1^{2}23$, $1^{3}2^{2}$, $1^{3}23$, $1^{2}2^{2}3$, $1^{4}23$, $1^{3}2^{2}3$, $1^{4}2^{2}3$, $1^{5}2^{2}3$,
$1^{4}2^{3}3$, $1^{5}2^{3}3$, $1^{5}2^{2}3^{2}$, $1^{5}2^{3}3^{2}$, $1^{6}2^{3}3^{2}$\\
Nr. $36$ with $25$ positive roots:\\
$1$, $2$, $3$, $12$, $13$, $1^{2}2$, $1^{2}3$, $123$, $1^{3}2$, $1^{2}23$, $1^{4}2$, $1^{3}2^{2}$, $1^{3}23$, $1^{4}23$, $1^{3}2^{2}3$, $1^{5}2^{2}$, $1^{5}23$, $1^{4}2^{2}3$,
$1^{5}2^{2}3$, $1^{6}2^{2}3$, $1^{7}2^{2}3$, $1^{6}2^{3}3$, $1^{7}2^{3}3$, $1^{8}2^{3}3$, $1^{8}2^{3}3^{2}$\\
Nr. $37$ with $25$ positive roots:\\
$1$, $2$, $3$, $12$, $13$, $1^{2}2$, $1^{2}3$, $123$, $1^{3}2$, $1^{2}23$, $1^{4}2$, $1^{3}23$, $1^{4}23$, $1^{3}2^{2}3$, $1^{5}23$, $1^{4}2^{2}3$, $1^{5}2^{2}3$, $1^{6}2^{2}3$,
$1^{7}2^{2}3$, $1^{6}2^{3}3$, $1^{7}2^{3}3$, $1^{8}2^{3}3$, $1^{7}2^{3}3^{2}$, $1^{8}2^{3}3^{2}$, $1^{9}2^{3}3^{2}$\\
Nr. $38$ with $25$ positive roots:\\
$1$, $2$, $3$, $12$, $13$, $1^{2}2$, $1^{2}3$, $12^{2}$, $123$, $1^{3}2$, $1^{2}23$, $12^{2}3$, $1^{3}23$, $1^{2}2^{2}3$, $1^{4}23$, $1^{3}2^{2}3$, $1^{4}2^{2}3$, $1^{3}2^{3}3$,
$1^{3}2^{2}3^{2}$, $1^{4}2^{3}3$, $1^{5}2^{3}3$, $1^{4}2^{3}3^{2}$, $1^{5}2^{3}3^{2}$, $1^{6}2^{3}3^{2}$, $1^{7}2^{4}3^{2}$\\
Nr. $39$ with $25$ positive roots:\\
$1$, $2$, $3$, $12$, $13$, $1^{2}2$, $1^{2}3$, $123$, $1^{3}2$, $1^{2}23$, $1^{3}2^{2}$, $1^{3}23$, $1^{2}2^{2}3$, $1^{4}23$, $1^{3}2^{2}3$, $1^{4}2^{2}3$, $1^{5}2^{2}3$,
$1^{4}2^{3}3$, $1^{5}2^{3}3$, $1^{5}2^{2}3^{2}$, $1^{6}2^{3}3$, $1^{5}2^{3}3^{2}$, $1^{6}2^{3}3^{2}$, $1^{7}2^{3}3^{2}$, $1^{7}2^{4}3^{2}$\\
Nr. $40$ with $26$ positive roots:\\
$1$, $2$, $3$, $12$, $13$, $1^{2}2$, $1^{2}3$, $123$, $1^{3}2$, $1^{2}23$, $1^{4}2$, $1^{3}2^{2}$, $1^{3}23$, $1^{4}23$, $1^{3}2^{2}3$, $1^{5}2^{2}$, $1^{5}23$, $1^{4}2^{2}3$,
$1^{5}2^{2}3$, $1^{6}2^{2}3$, $1^{7}2^{2}3$, $1^{6}2^{3}3$, $1^{7}2^{3}3$, $1^{8}2^{3}3$, $1^{7}2^{3}3^{2}$, $1^{8}2^{3}3^{2}$\\
Nr. $41$ with $26$ positive roots:\\
$1$, $2$, $3$, $12$, $13$, $1^{2}2$, $1^{2}3$, $123$, $1^{3}2$, $1^{2}23$, $1^{4}2$, $1^{3}2^{2}$, $1^{3}23$, $1^{4}23$, $1^{3}2^{2}3$, $1^{5}23$, $1^{4}2^{2}3$, $1^{5}2^{2}3$,
$1^{6}2^{2}3$, $1^{7}2^{2}3$, $1^{6}2^{3}3$, $1^{7}2^{3}3$, $1^{8}2^{3}3$, $1^{7}2^{3}3^{2}$, $1^{8}2^{3}3^{2}$, $1^{9}2^{3}3^{2}$\\
Nr. $42$ with $27$ positive roots:\\
$1$, $2$, $3$, $12$, $13$, $1^{2}2$, $1^{2}3$, $123$, $1^{3}2$, $1^{2}23$, $1^{4}2$, $1^{3}2^{2}$, $1^{3}23$, $1^{4}23$, $1^{3}2^{2}3$, $1^{5}2^{2}$, $1^{5}23$, $1^{4}2^{2}3$,
$1^{5}2^{2}3$, $1^{6}2^{2}3$, $1^{7}2^{2}3$, $1^{6}2^{3}3$, $1^{7}2^{3}3$, $1^{8}2^{3}3$, $1^{7}2^{3}3^{2}$, $1^{8}2^{3}3^{2}$, $1^{9}2^{3}3^{2}$\\
Nr. $43$ with $27$ positive roots:\\
$1$, $2$, $3$, $12$, $13$, $1^{2}2$, $1^{2}3$, $123$, $1^{3}2$, $1^{2}23$, $1^{4}2$, $1^{3}2^{2}$, $1^{3}23$, $1^{4}23$, $1^{3}2^{2}3$, $1^{5}23$, $1^{4}2^{2}3$, $1^{5}2^{2}3$,
$1^{6}2^{2}3$, $1^{5}2^{2}3^{2}$, $1^{7}2^{2}3$, $1^{6}2^{3}3$, $1^{7}2^{3}3$, $1^{8}2^{3}3$, $1^{7}2^{3}3^{2}$, $1^{8}2^{3}3^{2}$, $1^{9}2^{3}3^{2}$\\
Nr. $44$ with $27$ positive roots:\\
$1$, $2$, $3$, $12$, $13$, $1^{2}2$, $1^{2}3$, $123$, $1^{3}2$, $1^{2}23$, $1^{4}2$, $1^{3}2^{2}$, $1^{3}23$, $1^{4}23$, $1^{3}2^{2}3$, $1^{5}23$, $1^{4}2^{2}3$, $1^{5}2^{2}3$,
$1^{6}2^{2}3$, $1^{7}2^{2}3$, $1^{6}2^{3}3$, $1^{7}2^{3}3$, $1^{7}2^{2}3^{2}$, $1^{8}2^{3}3$, $1^{7}2^{3}3^{2}$, $1^{8}2^{3}3^{2}$, $1^{9}2^{3}3^{2}$\\
Nr. $45$ with $28$ positive roots:\\
$1$, $2$, $3$, $12$, $13$, $1^{2}2$, $1^{2}3$, $123$, $1^{3}2$, $1^{2}23$, $1^{4}2$, $1^{3}2^{2}$, $1^{3}23$, $1^{4}23$, $1^{3}2^{2}3$, $1^{5}2^{2}$, $1^{5}23$, $1^{4}2^{2}3$,
$1^{5}2^{2}3$, $1^{6}2^{2}3$, $1^{5}2^{2}3^{2}$, $1^{7}2^{2}3$, $1^{6}2^{3}3$, $1^{7}2^{3}3$, $1^{8}2^{3}3$, $1^{7}2^{3}3^{2}$, $1^{8}2^{3}3^{2}$, $1^{9}2^{3}3^{2}$\\
Nr. $46$ with $28$ positive roots:\\
$1$, $2$, $3$, $12$, $13$, $1^{2}2$, $1^{2}3$, $123$, $1^{3}2$, $1^{2}23$, $1^{4}2$, $1^{3}2^{2}$, $1^{3}23$, $1^{4}23$, $1^{3}2^{2}3$, $1^{5}23$, $1^{4}2^{2}3$, $1^{5}2^{2}3$,
$1^{6}2^{2}3$, $1^{5}2^{2}3^{2}$, $1^{7}2^{2}3$, $1^{6}2^{3}3$, $1^{7}2^{3}3$, $1^{8}2^{3}3$, $1^{7}2^{3}3^{2}$, $1^{8}2^{3}3^{2}$, $1^{9}2^{3}3^{2}$, $1^{9}2^{4}3^{2}$\\
Nr. $47$ with $28$ positive roots:\\
$1$, $2$, $3$, $12$, $13$, $1^{2}2$, $1^{2}3$, $123$, $1^{3}2$, $1^{2}23$, $1^{4}2$, $1^{3}2^{2}$, $1^{3}23$, $1^{4}23$, $1^{3}2^{2}3$, $1^{5}23$, $1^{4}2^{2}3$, $1^{5}2^{2}3$,
$1^{6}2^{2}3$, $1^{5}2^{2}3^{2}$, $1^{7}2^{2}3$, $1^{6}2^{3}3$, $1^{7}2^{3}3$, $1^{8}2^{3}3$, $1^{7}2^{3}3^{2}$, $1^{8}2^{3}3^{2}$, $1^{9}2^{3}3^{2}$, $1^{11}2^{4}3^{2}$\\
Nr. $48$ with $29$ positive roots:\\
$1$, $2$, $3$, $12$, $13$, $1^{2}2$, $1^{2}3$, $123$, $1^{3}2$, $1^{2}23$, $1^{4}2$, $1^{3}2^{2}$, $1^{3}23$, $1^{4}23$, $1^{3}2^{2}3$, $1^{5}2^{2}$, $1^{5}23$, $1^{4}2^{2}3$,
$1^{5}2^{2}3$, $1^{6}2^{2}3$, $1^{5}2^{2}3^{2}$, $1^{7}2^{2}3$, $1^{6}2^{3}3$, $1^{7}2^{3}3$, $1^{7}2^{2}3^{2}$, $1^{8}2^{3}3$, $1^{7}2^{3}3^{2}$, $1^{8}2^{3}3^{2}$,
$1^{9}2^{3}3^{2}$\\
Nr. $49$ with $29$ positive roots:\\
$1$, $2$, $3$, $12$, $13$, $1^{2}2$, $1^{2}3$, $123$, $1^{3}2$, $1^{2}23$, $1^{4}2$, $1^{3}2^{2}$, $1^{3}23$, $1^{4}23$, $1^{3}2^{2}3$, $1^{5}2^{2}$, $1^{5}23$, $1^{4}2^{2}3$,
$1^{5}2^{2}3$, $1^{6}2^{2}3$, $1^{5}2^{2}3^{2}$, $1^{7}2^{2}3$, $1^{6}2^{3}3$, $1^{7}2^{3}3$, $1^{8}2^{3}3$, $1^{7}2^{3}3^{2}$, $1^{8}2^{3}3^{2}$, $1^{9}2^{3}3^{2}$,
$1^{9}2^{4}3^{2}$\\
Nr. $50$ with $29$ positive roots:\\
$1$, $2$, $3$, $12$, $13$, $1^{2}2$, $1^{2}3$, $123$, $1^{3}2$, $1^{2}23$, $1^{4}2$, $1^{3}2^{2}$, $1^{3}23$, $1^{4}23$, $1^{3}2^{2}3$, $1^{5}2^{2}$, $1^{5}23$, $1^{4}2^{2}3$,
$1^{5}2^{2}3$, $1^{6}2^{2}3$, $1^{5}2^{2}3^{2}$, $1^{7}2^{2}3$, $1^{6}2^{3}3$, $1^{7}2^{3}3$, $1^{8}2^{3}3$, $1^{7}2^{3}3^{2}$, $1^{8}2^{3}3^{2}$, $1^{9}2^{3}3^{2}$,
$1^{11}2^{4}3^{2}$\\
Nr. $51$ with $30$ positive roots:\\
$1$, $2$, $3$, $12$, $13$, $1^{2}2$, $1^{2}3$, $123$, $1^{3}2$, $1^{2}23$, $1^{4}2$, $1^{3}2^{2}$, $1^{3}23$, $1^{4}23$, $1^{3}2^{2}3$, $1^{5}2^{2}$, $1^{5}23$, $1^{4}2^{2}3$,
$1^{5}2^{2}3$, $1^{6}2^{2}3$, $1^{5}2^{2}3^{2}$, $1^{7}2^{2}3$, $1^{6}2^{3}3$, $1^{7}2^{3}3$, $1^{7}2^{2}3^{2}$, $1^{8}2^{3}3$, $1^{7}2^{3}3^{2}$, $1^{8}2^{3}3^{2}$,
$1^{9}2^{3}3^{2}$, $1^{9}2^{4}3^{2}$\\
Nr. $52$ with $31$ positive roots:\\
$1$, $2$, $3$, $12$, $13$, $1^{2}2$, $1^{2}3$, $123$, $1^{3}2$, $1^{2}23$, $1^{4}2$, $1^{3}23$, $1^{5}2$, $1^{4}23$, $1^{6}2$, $1^{5}23$, $1^{4}2^{2}3$, $1^{6}23$, $1^{5}2^{2}3$,
$1^{7}23$, $1^{6}2^{2}3$, $1^{7}2^{2}3$, $1^{8}2^{2}3$, $1^{9}2^{2}3$, $1^{10}2^{2}3$, $1^{9}2^{3}3$, $1^{10}2^{3}3$, $1^{11}2^{3}3$, $1^{10}2^{3}3^{2}$, $1^{11}2^{3}3^{2}$,
$1^{12}2^{3}3^{2}$\\
Nr. $53$ with $31$ positive roots:\\
$1$, $2$, $3$, $12$, $13$, $1^{2}2$, $1^{2}3$, $123$, $1^{3}2$, $1^{2}23$, $1^{4}2$, $1^{3}2^{2}$, $1^{3}23$, $1^{4}23$, $1^{3}2^{2}3$, $1^{5}2^{2}$, $1^{5}23$, $1^{4}2^{2}3$,
$1^{5}2^{2}3$, $1^{6}2^{2}3$, $1^{5}2^{2}3^{2}$, $1^{7}2^{2}3$, $1^{6}2^{3}3$, $1^{7}2^{3}3$, $1^{7}2^{2}3^{2}$, $1^{8}2^{3}3$, $1^{7}2^{3}3^{2}$, $1^{8}2^{3}3^{2}$,
$1^{9}2^{3}3^{2}$, $1^{9}2^{4}3^{2}$, $1^{11}2^{4}3^{2}$\\
Nr. $54$ with $34$ positive roots:\\
$1$, $2$, $3$, $12$, $13$, $1^{2}2$, $1^{2}3$, $123$, $1^{3}2$, $1^{2}23$, $1^{4}2$, $1^{3}2^{2}$, $1^{3}23$, $1^{4}23$, $1^{3}2^{2}3$, $1^{5}2^{2}$, $1^{5}23$, $1^{4}2^{2}3$,
$1^{5}2^{2}3$, $1^{6}2^{2}3$, $1^{5}2^{3}3$, $1^{7}2^{2}3$, $1^{6}2^{3}3$, $1^{7}2^{3}3$, $1^{8}2^{3}3$, $1^{7}2^{3}3^{2}$, $1^{8}2^{4}3$, $1^{8}2^{3}3^{2}$, $1^{9}2^{4}3$,
$1^{9}2^{3}3^{2}$, $1^{9}2^{4}3^{2}$, $1^{11}2^{4}3^{2}$, $1^{11}2^{5}3^{2}$, $1^{12}2^{5}3^{2}$\\
Nr. $55$ with $37$ positive roots:\\
$1$, $2$, $3$, $12$, $13$, $1^{2}2$, $1^{2}3$, $123$, $1^{3}2$, $1^{2}23$, $1^{4}2$, $1^{3}2^{2}$, $1^{3}23$, $1^{4}23$, $1^{3}2^{2}3$, $1^{5}2^{2}$, $1^{5}23$, $1^{4}2^{2}3$,
$1^{5}2^{2}3$, $1^{6}2^{2}3$, $1^{5}2^{3}3$, $1^{7}2^{2}3$, $1^{6}2^{3}3$, $1^{7}2^{3}3$, $1^{8}2^{3}3$, $1^{7}2^{3}3^{2}$, $1^{9}2^{3}3$, $1^{8}2^{4}3$, $1^{8}2^{3}3^{2}$,
$1^{9}2^{4}3$, $1^{9}2^{3}3^{2}$, $1^{10}2^{4}3$, $1^{9}2^{4}3^{2}$, $1^{11}2^{4}3^{2}$, $1^{11}2^{5}3^{2}$, $1^{12}2^{5}3^{2}$, $1^{13}2^{5}3^{2}$
\end{tiny}

\end{appendix}

\bibliographystyle{amsalpha}
\bibliography{quantum}

\end{document}